\documentclass[11pt]{amsart}
\usepackage{float,url}
\usepackage{amssymb}
\theoremstyle{plain}
\newtheorem{thm}{Theorem}[section]
\newtheorem{lemma}[thm]{Lemma}

\newtheorem{prop}[thm]{Proposition}

\theoremstyle{definition}

\newcommand{\fra}{\mathfrak{a}}

\newcommand{\frg}{\mathfrak{g}}
\newcommand{\frh}{\mathfrak{h}}

\newcommand{\frk}{\mathfrak{k}}

\newcommand{\frp}{\mathfrak{p}}

\newcommand{\frt}{\mathfrak{t}}

\newcommand{\bbZ}{\mathbb{Z}}

\newcommand{\caC}{\mathcal{C}}
\newcommand{\caD}{\mathcal{D}}

\newcommand{\caF}{\mathcal{F}}

\newcommand{\caM}{\mathcal{M}}

\newcommand{\caT}{\mathcal{T}}
\newcommand{\caU}{\mathcal{U}}

\begin{document}

\title[An enhanced Helgason-Johnson bound for $Sp(p, q)$]
{An enhanced Helgason-Johnson bound for $Sp(p, q)$}

\author{Zhan Ying}
\address[Ying]{School of Mathematical Sciences, Soochow University, Suzhou 215006,
P.~R.~China}
\email{zhanying\_1899@163.com}

\author{Chao-Ping Dong}
\address[Dong]{School of Mathematical Sciences, Soochow University, Suzhou 215006,
P.~R.~China}
\email{chaopindong@163.com}

\abstract{Under the assumption that the irreducible unitary representation $\pi$ is infinite dimensional, this paper improves the Helgason-Johnson bound in 1969 for $Sp(p, q)$.}

\endabstract



\maketitle


\section{Introduction}
Take positive integers $ p\le q$ and let $n:=p+q$. In this paper, we set $G=Sp(p, q)$, which has a maximal compact subgroup $K:= Sp(p)\times Sp(q)$. Let $\theta$ be the Cartan decomposition of $G$ so that $K=G^{\theta}$. Let $T$ be a maximal torus of $K$.  We have Cartan decomposition
$$
\frg_0=\frk_0 + \frp_0
$$
on the Lie algebra level. Let $\frg$ (resp. $\frk, \frt$) be the complexifications of $\frg_0$ (resp. $\frk_0, \frt_0$).

We fix 	once for all that
$$
\Delta^+(\frk, \frt)=\{e_i\pm e_j\mid\mbox{$1\le i<j\le p$ or $p+1\le i<j\le n$}\}\cup\{2e_i\mid 1\le i\le n\}.
$$
The half sum of its elements is
$$
\rho_c=(p,p-1,\dots,1\mid q,q-1,\dots,1).
$$
By a \emph{$K$-type}, we mean an irreducible representation of $K$. We will denote a $K$-type by $E_{\eta}$, where $\eta$ is its highest weight. There are ${n \choose p}$ ways of choosing a positive root system $\Delta^+(\frg, \frt)$ containing $\Delta^+(\frk, \frt)$. Among them, we fix
$$
\Delta^+(\frg, \frt)=\Delta^+(\frk, \frt) \cup \Delta^+(\frp,\frt),
$$
where
$$
\Delta^+(\frp,\frt)=\{e_i\pm e_j\mid 1\le i\le p, p+1\le j\le n\}.
$$
Denote by $\rho$ the half sum of the roots in $\Delta^+(\frg, \frt)$. Then
$$
\rho=(n,n-1,\dots,q+1\mid q,q-1,\dots,1).
$$
In particular, $\|\rho\|=\sqrt{n(n+1)(2n+1)/6}$.

Let $\mathcal{C}$ (resp. $\mathcal{C}_\frg$) be the dominant Weyl chamber corresponding to $\Delta^+(\frk,\frt)$ (resp. $\Delta^+(\frg,\frt)$). Since $\mathcal{C}_\frg\subset \mathcal{C}$, we define
\[W(\frg,\frt)^1=\{w\in W(\frg,\frt)\mid w(\mathcal{C}_\frg)\subset \mathcal{C}\}.\]
In particular, we denote by $\mathcal{C}_w$ the Weyl chamber $w(\mathcal{C}_\frg)$. For $w\in W(\frg,\frt)^1$,
\[w\Delta^+(\frg,\frt)=\Delta^+(\frk,\frt)\cup w\Delta^+(\frp,\frt).\]
Here $w\in W(\frg,\frt)^1$ acts on $(z_1, \dots, z_n)\in\frt^*$ as follows: it selects $p$ entries from $z_1, \dots, z_n$, arranges the selected ones in descending order as $(u_1,\dots,u_p)$, and writes the remaining components as $(v_1,\dots,v_q)$ in their original relative order.
Then
\begin{equation}\label{wgt-1}
	w(z_1,\dots,z_n)=(u_1,\dots,u_p\mid v_1,\dots,v_q).
\end{equation}
Put $\rho^{(w)}:=w\rho$. Denote by $\rho_n^{(w)}$ the half sum of the roots in $w\Delta^+(\frp,\frt)$, then
\[\rho_n^{(w)}=\rho^{(w)}-\rho_c.\]

Given any $\eta\in\mathcal{C}$, we can find $w\in W(\frg, \frt)^1$ such that $\eta+2\rho_c\in \mathcal{C}_{w}$. As stated in \cite{SV}, there exists a unique
vector within $\mathcal{C}_{w}$ which is closest to $\eta+2\rho_c-\rho^{(w)}$. Denote it by $P(\eta+2\rho_c-\rho^{(w)})$. It turns out that it is independent of the choice of $w$. Let us denote it by $\lambda_a(\eta)$. Note that
$\|\lambda_a(\eta)\|$ is the \emph{lambda norm} of $\eta$ defined by Vogan in \cite{V81}.

Let $\pi$ be irreducible admissible $(\frg,K)$ module. Take one of its lambda lowest $K$-types $E_{\mu}$. Let $G(\lambda_a(\mu))$ be the isotropic group of $\lambda_a(\mu)$ under the $G$ action. Let $H_{\rm s}=T_{\rm s} A_{\rm s}$ be a maximally split $\theta$-stable Cartan subgroup of $G(\lambda_a(\mu))$.
Then the infinitesimal character of  $\pi$ has the form
\begin{equation}\label{inf-char}
	\Lambda=\big(\lambda_a(\mu),\nu\big)\in \frh_{\rm s}^*=\frt_{\rm s}^*+\fra_{\rm s}^*.
\end{equation}
The Helgason-Johnson bound derived in 1969 says that
$$
\|\nu\|\leq \|\rho\|
$$
whenever $\pi$ is unitary, see \cite{HJ}.

Earlier we have sharpened the Helgason-Johnson bound for all linear exceptional Lie groups \cite{D23}. The current paper focuses on $Sp(p, q)$.
Put
\begin{equation}\label{B0}
	B_0:=\dfrac{4p^3+12pq(q-1)-p}{6}+4\delta_{p,q},
\end{equation}
where $\delta_{p,q}$ denotes the Kronecker delta; that is
\[\delta_{p,q}=\left\{
\begin{array}{ll}
	1, & \text{if } q=p,\\
	0, & \text{otherwise}
\end{array}
\right.\]
The meaning of $B_0$ will be explained in \eqref{B0-explanation}.

\begin{thm}\label{thm-main}
	Let $\pi$ be any irreducible unitary $(\frg, K)$ module for $Sp(p,q)$ which is infinite dimensional. Assume that its infinitesimal character is given by \eqref{inf-char}. Then
	\[\Vert\nu\Vert\le \left\{
	\begin{array}{ll}
		B_0+\dfrac{3}{2}, & \text{if } q=p+1,\\
		B_0, & \text{otherwise}
	\end{array}
	\right.\]
\end{thm}

Note that when $q=p$,
$$
\|\rho\|^2=\frac{p}{3}+2 p^2+\frac{8 p^3}{3}
$$
while our bound reads as
$$
4-\frac{p}{6}-2 p^2+\frac{8 p^3}{3}
$$
It is not a huge improvement. However, when $q=k p$ for some positive integer $k>1$,
$$
\|\rho\|^2=\frac{p}{6}+\frac{k p}{6}+\frac{p^2}{2}+k p^2+\frac{k^2 p^2}{2}+\frac{p^3}{3}+k p^3+k^2 p^3+\frac{k^3 p^3}{3}
$$
while our bound reads as
$$
-\frac{p}{6}-2 k p^2+\frac{2 p^3}{3}+2 k^2 p^3.
$$
The improvement is noticeable now.

Since $Sp(p, q)$ has no minimal representation, our bound is not tight. But we believe that it is close to be.
For instance, consider $Sp(3, 3)$. The trivial representation has the following \texttt{atlas} \cite{ALTV,At} final parameter:
$$
(679, [6,5,4,3,2,1], [11,11,7,7,3,3]/2).
$$
The last vector is $\nu$ of the trivial representation. Note that the Helgason-Johnson bound is $\|\rho\|=\sqrt{91}$, while $\|\nu\|=\sqrt{89.5}$. On the other hand, our bound is $\sqrt{57.5}$. Let $\pi$ be the irreducible representation of $Sp(3, 3)$ with the following final parameter:
$$
(634, [6,4,5,3,2,1], [9,0,9,5,5,0]/2)
$$
\texttt{atlas} tells us that $\pi$ is infinite-dimensional and unitary.
One computes that $\|\nu(\pi)\|=\sqrt{53}$.

Vogan addressed the question in 2020 \cite{V20-1, V20-2} by using Parthasarathy's \emph{Dirac operator inequality} \cite{Pa72, Pa80}. Here Dirac inequality is also a major tool for us: we use it in the compact way as \emph{spin norm} (see \S \ref{sec-spin-norm} for its precise definition), which also builds in the knowledge of PRV component \cite{PRV}.


Now let us mention the outline of the current paper. Section \ref{sec-pre} collects necessary preliminaries. Section \ref{sec-lambda-norm} gives an explicit formula for the lambda norm of $K$-types of $Sp(p, q)$, which is the first new ingredient. It has been obtained by the first named author in his 2011 thesis that
\begin{equation}\label{spin-lambda}
\|\mu\|_{\rm spin} \ge \|\mu\|_{\rm lambda}.
\end{equation}
See \cite{D13}. Then two natural questions emerge: (i) when shall we have equality in \eqref{spin-lambda}?
(ii) to which extent can we understand the following difference?
$$
\|\mu\|_{\rm spin}^2 - \|\mu\|_{\rm lambda}^2
$$
In 2016, a representation-theoretic answer to (i) was given in \cite{DD16} based on Vogan bijection.  On the other hand, (ii) has been a long standing perplexity for us. In Section \ref{sec-convexity}, we find certain convexity of the above difference.  This is the second new ingredient, which allows us to use Jensen inequality among some integral vectors. After doing necessary preparations in Section \ref{sec-RT}, we prove Theorem \ref{thm-main} for u-small (resp., u-large) $K$-types in Section \ref{sec-u-small} (resp., Section \ref{sec-u-large}).

\section{Preliminaries}\label{sec-pre}
Recall that we have fixed
$$
\Delta^+(\frg, \frt)=\Delta^+(\frk, \frt) \cup \Delta^+(\frp, \frt)
$$
in the introduction. Then $e_t-e_{t+1}(1\le t\le n-1)$ and $2e_n$ are the simple roots of $\Delta^+(\frg,\frt)$, with the corresponding fundamental dominant weights being $\xi_j=\sum_{t=1}^j e_t$, $1\le j\le n$. Note that $\beta:=e_1+e_{p+1}$ is the highest weight of $\frp$.

Let $\pi$ be an \emph{infinite dimensional} irreducible representation of $G$. Let $\mu$ be (the highest weight of) a $K$-type occurring in $\pi$. By \cite{V80},  the $K$-types of $\pi$ is a union of \emph{Vogan pencils}:
\begin{equation}
	\mathrm{Pencil}(\mu):=\{\mu+m\beta\mid m\in \bbZ_{\ge 0}\}.
\end{equation}

Put
\[B(\mu):=\min_{m\in\bbZ_{\ge 0}}\Vert \mu+m\beta\Vert_{\mathrm{spin}}-\Vert\mu\Vert_{\mathrm{lambda}},\]
where $\Vert\cdot\Vert_{\mathrm{spin}}$ and $\Vert\cdot\Vert_{\mathrm{lambda}}$ stand for spin norm and lambda norm, respectively. The notion spin norm will be introduced in \S\ref{sec-spin-norm} soon.

\subsection{Spin norm}\label{sec-spin-norm}
Let $\caC$ (resp. $\caC_\frg$) be the dominant Weyl chamber corresponding to $\Delta^+(\frk,\frt)$ (resp. $\Delta^+(\frg,\frt)$). Since $\caC_\frg\subset \caC$, we define
\[W(\frg,\frt)^1=\{w\in W(\frg,\frt)\mid w(\caC_\frg)\subset \caC\}.\]
In particular, we denote by $\caC_w$ the Weyl Chamber $w(\caC_\frg)$.

Since $\frk$ has no center, then $\Delta(\frp,\frt)$ is a root system. For $w\in W(\frg,\frt)^1$,
\[w\Delta^+(\frg,\frt)=\Delta^+(\frk,\frt)\cup w\Delta^+(\frp,\frt).\]
As in \cite{DY}, $w\in W(\frg,\frt)^1$ acts as follows: from the $n$ components $(z_1,\dots,z_n)$, it selects $p$ of them, arranges these selected ones in descending orders as $(u_1,\dots,u_p)$, and writes the remaining compents, preserving their original relative order, as $(v_1,\dots,v_q)$. Then
\begin{equation}\label{wgt1}
	w(z_1,\dots,z_n)=(u_1,\dots,u_p\mid v_1,\dots,v_q).
\end{equation}

 Denote by $\rho_n^{(w)}$ the half sum of the roots in $w\Delta^+(\frp,\frt)$, then
 \[\rho_n^{(w)}=w\rho-\rho_c.\]
  Let $\mu$ be the highest weight of a $K$-type. Then the spin norm of $\mu$ is defined to be
 \begin{equation}\label{spin}
 	\Vert\mu\Vert_{\mathrm{spin}}:=\min_{w\in W(\frg,\frt)^1}\Vert \{\mu-\rho_n^{(w)}\}+\rho_c\Vert,
 \end{equation}
 where $\{\mu-\rho_n^{(w)}\}$ is the unique dominant weight that is conjugate to $\mu-\rho_n^{(w)}$ under the action of $W(\frk,\frt)$.

 Assume that $u_1, \dots, u_r$ are real numbers. Let $u_{(1)}\ge\cdots\ge u_{(r)}$ be the decreasing rearrangement of the absolute values of  $u_1, \dots, u_r$. Define
  $$\{u_1, \dots, u_r\}_r=(u_{(1)}, \dots, u_{(r)})$$
  and
 \begin{equation}\label{Phi}
 	\Phi_r(u_1,\dots,u_r)=\sum_{t=1}^r(u_{(t)}+r-t+1)^2.
 \end{equation}

 For any weight $\eta=(a_1, \dots, a_p\mid b_1, \dots, b_q)$, define
 \begin{equation}\label{k-norm}
 \Vert\eta\Vert_\frk^2:=\Vert\{\eta\}+\rho_c\Vert^2=\Phi_p(a_1, \dots, a_p)+\Phi_q(b_1, \dots, b_q).
 \end{equation}
Note that based on \cite{HP}, the spin norm was introduced in \cite{D13}. It offers a compact way of expressing the Dirac inequality.

 \begin{lemma}\label{lem-Phi}
 	Suppose that $u_1\ge\cdots\ge u_r$ and $v_1\ge\cdots\ge v_r$ are two nonnegative real number sequences. If
 	\[\sum_{t=1}^l u_t\le \sum_{t=1}^lv_t,\quad 1\le l\le r,\]
 	$\Phi_r(u_1,\dots,u_r)\le\Phi_r(v_1,\dots,v_r)$.
 \end{lemma}
 \begin{proof}
 	Set $u_{r+1}=0$. Define
 	\[U_l=\sum_{t=1}^l u_t,\quad V_l=\sum_{t=1}^l v_t.\]
 	From $U_l\le V_l$ and $u_l-u_{l+1}\ge 0$, we have
 	\[\sum_{t=1}^ru_t^2=\sum_{l=1}^rU_l(u_l-u_{l+1})\le\sum_{l=1}^rV_l(u_l-u_{l+1})=\sum_{t=1}^ru_tv_t.\]
 	Using $2u_tv_t\le u_t^2+v_t^2$, we obtain
 	\[\sum_tu_t^2\le \sum_{t} v_t^2.\]
 	Notice that
 	\[\sum_{t=1}^r(r-t+1)u_t=\sum_{l=1}^rU_l,\quad \sum_{t=1}^r(r-t+1)v_t=\sum_{l=1}^rV_l,\]
 	then
 	\begin{align*}
 	\Phi_r(u_1,\dots,u_r)&=\sum_{t=1}^r(u_{t}+r-t+1)^2=\sum_{t=1}^ru_t^2+2\sum_{t=1}^r(r-t+1)u_t+(r-t+1)^2\\
 	&=\sum_{t=1}^ru_t^2+2\sum_{l=1}^rU_l+(r-t+1)^2\\
 	&\le \sum_{t=1}^rv_t^2+2\sum_{l=1}^rV_l+(r-t+1)^2=\Phi_r(v_1,\dots,v_r).
 \end{align*}
 \end{proof}
  For nonnegative integer vectors $u=(u_1,\dots,u_p\mid u_{p+1},\dots,u_n)$ and $v=(v_1,\dots,v_p\mid v_{p+1},\dots,v_n)$, we say that $u$  \emph{weakly dominates} $v$ if $u_t\ge v_t$ for all $1\le t\le n$. In this case, we write $u\gg v$. Lemma \ref{lem-Phi} shows that if $v\gg u$, then
  \begin{equation}\label{weakly-dominate}
  \Vert v\Vert_\frk^2=\Phi_p(v_1,\dots,v_p)+\Phi_q(v_{p+1},\dots,v_n)\ge \Phi_p(u_1,\dots,u_p)+\Phi_q(u_{p+1},\dots,u_n)=\Vert u\Vert_\frk^2.
  \end{equation}

\subsection{Unitarily small $K$-types}In \cite{SV}, Salamanca-Riba and Vogan introduced the concept of the \emph{unitarily small convex hull}, which is defined as
\[R(\Delta\big(\frp,\frt)\big):=\left\{\sum_{\alpha\in\Delta(\frp,\frt)}b_\alpha\alpha\mid 0\le b-\alpha\le 1\right\}.\]
 We say that $\mu$ is \emph{u-small}, if it lies within the unitarily small convex hull; otherwise, we say that $\mu$ is \emph{u-large}. Let $\mathcal{U}_{p,q}$ denote the set of all u-small $\Delta^+(\frk,\frt)$-dominant weights and $\mathcal{U}_{p,q}^c$ denote the the set of all u-large $\Delta^+(\frk,\frt)$-dominant weights.

  Put
 \begin{equation}\label{T}
 	\mathcal{T}(\mu):=\max\{m\in\bbZ_{\ge 0}\mid \mu+m\beta\in\mathcal{U}_{p,q}\},\quad \mu\in \mathcal{U}_{p,q}
 \end{equation}
 \emph{Throughout this paper}, we set
 \begin{equation}\label{mu}
 	\mu=(a_1, \dots, a_p \mid b_1, \dots, b_q)
 \end{equation}
 as the highest weight of a $K$-type. We may also write $a=(a_1, \dots, a_p)$ and $b=(b_1, \dots, b_q)$.

\begin{lemma}[Theorem 6.7(d) of \cite{SV}]\label{usmall-1}
	Let $\mu$ be any $\Delta^+(\frk, \frt)$-dominant weight and $w\in W(\frg, \frt)^1$. If $\mu\in \mathcal{C}_w$, then $\mu \in\mathcal{U}_{p, q}$ if and only if $\langle\mu+2\rho_c, w\xi_t \rangle\le 2\langle \rho, \xi_t \rangle$ for all $1\le t\le n$.
\end{lemma}
\begin{lemma}[Proposition 2.2 of \cite{DY}]\label{usmall-2}
Take $\mu$ as in \eqref{mu}. Assume in addition that $\mu$ is in $\caU_{p,q}$. Then $\mu$ is u-small if and only if
	\begin{equation}\label{usmall-2.1}
		\sum_{i=1}^f a_i+\sum_{j=1}^g b_j\le 2pq-2(p-f)(q-g)
	\end{equation}
	for any $0\le f\le p$ and $0\le g\le q$.
\end{lemma}
We record here some facts that will be used later.
\begin{lemma}\label{m}
Take $\mu$ as in \eqref{mu}. Assume in addition that $\mu$ is in $\caU_{p,q}$. Then
	\begin{itemize}
		\item[(a)] $a_1\le 2q$ and $b_1\le 2p$.
		\item[(b)] If $a_1< q$ and $b_1< p$, we have $\min\{q-a_1,p-b_1\}\le \mathcal{T}(\mu)$.
	\end{itemize}
\end{lemma}
\begin{proof}
	Setting $f=0, g=1$ and $f=1, g=0$ in \eqref{usmall-2.1} yields (a). For (b), consider $\nu=(q,\dots,q\mid p,\dots,p)$. Notice that for all $0\le f\le p$ and $0\le g\le q$,
	\[fq+gp\le 2(fq+gp)-2fg=2pq-2(p-f)(q-g).\]
	Then Lemma \ref{usmall-2} shows that $\nu$ is u-small. Now take $m=\min\{q-a_1, p-b_1\}$; then the fact that $\nu$ is u-small implies that $\mu+m\beta$ is also u-small. This proves (b).
\end{proof}
\begin{lemma}\label{large}
	Take $\mu$ as in \eqref{mu}. Assume in addition that $\mu$ is in $\caU_{p,q}^c$ and $w\in W(\frg,\frt)^1$. Write
	\[\rho_n^{(w)}=(k_1, \dots, k_p \mid r_1, \dots, r_q).\]
	 If $\mu \in \caC_w$, then either $a_i>2k_i$ for some $1\le i\le p$, or $b_j>2r_j$ for some $1\le j\le q$.
\end{lemma}
\begin{proof}
	Since $\mu$ is u-large, then Lemma \ref{usmall-1}  implies theta exists $1\le t\le n$ such that
$$
\langle \mu+2\rho_c, w\xi_t\rangle>2\langle \rho, \xi_t\rangle,
$$
which is equivalent to
\[\langle\mu-2\rho_n^{(w)}, w\xi_t\rangle>0.\]
	Since the aforementioned inner product is positive, this implies that at least one of the quantities $a_i-2k_i$ or $b_j-2r_j$ is greater than $0$. This proves the lemma.
\end{proof}
\subsection{The structure of $\rho_n^{(w)}$}
Define
\[\Omega_{p,q}=\{\rho_n^{(w)}\mid w\in W(\frg,\frt)^1\}.\]
Then \cite{DY} shows that $\Omega_{p,q}$ is the set consisting of $(u_1,\dots,u_p\mid v_1,\dots,v_q)\in \bbZ^n$ such that
\begin{itemize}
	\item[(a)] $q\ge u_1\ge\cdots\ge u_p\ge 0$, $p\ge v_1\ge\cdots\ge v_q\ge 0$;
	\item[(b)] $\sum\limits_{i=1}^p u_i+\sum\limits_{j=1}^q v_j=pq$;
	\item[(c)] $v_j=\mathrm{card}\{i\mid q-u_i\ge j\}=\max\{i\mid q-u_i\ge j\}$ for any $1\le j\le q$.
\end{itemize}

Let us write a typical element in $\Omega_{p,q}$ as
\begin{equation}\label{rho-n-w}
\rho_n^{(w)}=(K_w\mid R_w)=(k_1,\dots,k_p\mid r_1,\dots,r_q).
\end{equation}
By convention, $k_0=q$,\ $r_0=p$ and $k_{p+1}=r_{q+1}=0$. For $1\le i\le q+1$, let
\[K_w(i):=(\underbrace{q-i+1,\dots,q-i+1}_{r_{i-1}-r_i})\]
if $r_{i-1}-r_i\neq 0$; let $K_w(i)$ be the empty array otherwise. For $1\le j\le p+1$, let
\[R_w(j):=(\underbrace{p-j+1,\dots,p-j+1}_{k_{j-1}-k_j})\]
if $k_{j-1}-k_j\neq 0$; let $R_w(j)$ be the empty array otherwise. Then
\[\rho_n^{(w)}=\big(K_w(1),\dots,K_w(p+1)\mid R_w(1),\dots,R_w(q+1)\big).\]

 \begin{lemma}\label{omegapq}
 	Let $(x_1, \dots, x_p \mid y_1, \dots, y_q)$ and $(k_1, \dots, k_p\mid r_1, \dots, r_q)$ be in $\Omega_{p, q}$. Then
 	\[\sum_{i=1}^p |x_i-k_i|=\sum_{j=1}^q |y_j-r_j|.\]
 \end{lemma}
 \begin{proof}
 Consider Young diagrams with row lengths $q-x_p,\dots,q-x_1$ and $q-k_p,\dots,q-k_1$, respectively. Then the column lengths of these two Young diagrams are $y_1,\dots,y_q$ and $r_1,\dots,r_q$, respectively. Consider the number of boxes that appear in exactly one of the two Young diagrams. On the one hand, counting by rows, we obtain
 		\[\sum_{i=1}^p|(q-x_i)-(q-k_i)|=\sum_{i=1}^p|x_i-k_i|.\]
 	On the other hand, counting by columns, we obtain
 		\[\sum_{j=1}^q|y_j-r_j|.\]
 	Therefore
 	\[\sum_{i=1}^p|x_i-k_i|=\sum_{j=1}^q|y_j-r_j|.\]
 \end{proof}

\section{Lambda norm for $Sp(p, q)$: an explicit formula}\label{sec-lambda-norm}
Recall that the lambda norm of a $K$-type has been defined in the introduction.
 This section aims to give an explicit form of the lambda norm in the $Sp(p, q)$ case.

 \begin{lemma}\label{lambda-1}
  Let $G$ be $Sp(p, q)$ and $\mu$ be a $\Delta^+(\frk,\frt)$-dominant weight. Choose $\widetilde{w}\in W(\frg,\frt)^1$ such that $\mu+2\rho_c\in C_{\widetilde{w}}$. Then
 \[\lambda_a(\mu)=\mu+2\rho_c-\dfrac{1}{\mathrm{card}(W_{\mu+2\rho_c})}\sum_{w\in W_{\mu+2\rho_c}}w\cdot \rho^{(\widetilde{w})},\]
 where $W_{\mu+2\rho_c}$ is the stabilizer of $\mu+2\rho_c$ in $W(\frg,\frt)$.
 \end{lemma}
 \begin{proof}
Let $\Pi$ be the set of simple roots for $\Delta^+(\frg,\frt)$. Then
 	 \[\Pi=\{e_i-e_{i+1}\mid\mbox{$1\le i\le n-1$}\}\cup\{2e_n\}.\]
Write $\Delta^+=\widetilde{w}\Delta^+(\frg,\frt)$ and $W=W(\frg,\frt)$. Note that $\Delta^+$ has simple roots $\widetilde{w}\Pi$.

Put $v=\mu+2\rho_c$ and $\gamma=\rho^{(\widetilde{w})}$. Set
 	\[
 	\gamma_0=\dfrac{1}{\mathrm{card}(W_v)}\sum_{w\in W_v}w\cdot \gamma.
 	\]
 	Let $\Delta_0$ collect all the the roots which are orthogonal to $v$. Then $\Delta_0$ is a root system. Put $\Delta_0^+:=\Delta_0\cap \Delta^+$. If $\alpha\in\Delta_0$, then the above definition shows that $s_\alpha(\gamma_0)=\gamma_0$, where $s_\alpha$ is the simple reflection corresponding to $\alpha$. Hence $\langle\gamma_0,\alpha\rangle=0$. Therefore $\gamma_0$ is $\Delta_0^+$-dominant.
 	
 	 Note that $\gamma_0$ lies in the convex hull of the $W_v$-orbit of $\gamma$. Applying Proposition 1.7 of \cite{SV} to $\Delta_0$ and $\gamma_0$, we obtain that
 	 \[\gamma-\gamma_0=\sum_{\alpha\in\Pi_0}k_\alpha\alpha,\]
 where $\Pi_0$ is the set of simple roots for $\Delta_0^+$ and $k_\alpha\ge 0$. Note that $\Pi_0=\widetilde{w}\Pi\cap \Delta_0$.
  For any simple root $\beta'$ of $\Delta^+$ which is not in $\Delta_0^+$, we have
 	 \[\langle\gamma_0, \beta'\rangle=\langle\gamma, \beta' \rangle-\sum_{\alpha\in\Pi_0}k_\alpha \langle\alpha, \beta'\rangle.\]
 Note that  $\langle \alpha, \beta'\rangle\leq 0$. Thus $\langle \gamma_0,\beta' \rangle\ge 0$. Hence $\gamma_0\in C_{\widetilde{w}}$.
 	
 	 Observe that if $\sigma\in W_v$, then $\sigma(v-\gamma_0)=v-\gamma_0$. Set $v_0=v-\gamma_0$; this shows that
 	 \[
 	 W_v\subset W_{v_0}.
 	 \]
 	 Therefore the fact that $\gamma_0$ lies in the convex hull of the $W_v$-orbit of $\gamma$ implies that $\gamma_0$ lies in the convex hull of the $W_{v_0}$-orbit of $\gamma$.
 	
 	 Finally, we \emph{claim} that for $\frg_0=\mathfrak{sp}(p,q)$, the element $v_0=v-\gamma_0\in C_{\widetilde{w}}$. Indeed, let $\overline{v}=\big(\widetilde{w}\big)^{-1}v$; then $\overline{v}$ is in fact the rearrangement of $v$ in decreasing order. Rewrite
 	 	\[\gamma_0=\dfrac{1}{\mathrm{card}(W_v)}\sum_{w\in W_v}w\cdot \gamma=\widetilde{w}\cdot\dfrac{1}{\mathrm{card}(W_v)}\sum_{w\in W_v}\big(\widetilde{w}\big)^{-1}w\widetilde{w}\cdot\rho.\]
 	 Observe that
 	 \[\big(\widetilde{w}\big)^{-1}W_v\widetilde{w}=W_{\overline{v}},\]
 	 Thus
 	 \[\overline{\gamma_0}:=\big(\widetilde{w}\big)^{-1}\cdot\gamma_0=\dfrac{1}{\mathrm{card}(W_{\overline{v}})}\sum_{w\in W_{\overline{v}}}w\cdot\rho.\]
 	 To verify that $v_0$ is $\Delta^+$-dominant, it suffices to show that for every simple root $\alpha$ of $\Delta^+(\frg,\frt)$,
 	 	\[\langle v_0,\widetilde{w}\cdot\alpha\rangle=\langle \widetilde{w}^{-1}\cdot v-\widetilde{w}^{-1}\cdot \gamma_0,\alpha\rangle=\langle \overline{v}-\overline{\gamma_0},\alpha\rangle\ge 0.\]
 	
 	 Now we show that  any number appears at most twice in $\overline{v}$. Indeed, this holds for $v=\mu+2\rho_c$ and thus for $\bar{v}$ as well. Therefore, for each pair of indices $i,i+1$ corresponding to a repeated component in $\overline{v}$, the $i$-th and $(i+1)$-th components of $\overline{\gamma_0}$ are the averages of the corresponding components of $\rho$, while all other components are identical to those of $\rho$. Hence we find that whenever $\langle\overline{v}, \alpha\rangle=0$, we necessarily have $\langle\overline{\gamma_0}, \alpha\rangle=0$. On the other hand, when $\langle\overline{v}, \alpha\rangle>0$, we have
 	 \[\dfrac{2\langle \overline{v},\alpha\rangle}{\langle\alpha,\alpha\rangle}\ge 1=\dfrac{2\langle\overline{\gamma_0},\alpha\rangle}{\langle\alpha,\alpha\rangle}\quad \mathrm{or}\quad \dfrac{2\langle \overline{v},2e_n \rangle}{\langle 2e_n,2e_n \rangle}\ge 2>\dfrac{3}{2}\ge \dfrac{\langle\overline{\gamma_0},2e_n\rangle}{\langle 2e_n,2e_n\rangle}.\]
 	 Thus $\langle\overline{v}-\overline{\gamma_0},\alpha\rangle\ge 0$ and $v_0\in C_{\widetilde{w}}$. Therefore, the claim holds. Now the conclusion follows from applying Corollary 1.9 (b) of \cite{SV} to $v_0$ and $\gamma$.
 \end{proof}

 The above lemma was inspired by Wong \cite{W24} in the $SU(p, q)$ case. Now let us move on to give an explicit formula for the lambda norm.
 Let $\mu=(a_1,\dots,a_p\mid b_1,\dots,b_q)$ be a $\Delta^+(\frk,\frt)$-dominant weight. For convenience, put
 \begin{equation}\label{L-mu}
 L(\mu)=\Vert\mu\Vert_{\mathrm{lambda}}^2.
 \end{equation}
 \begin{lemma}\label{lambda-2}
 	Let $G=Sp(p,q)$ and take $\mu$ as in \eqref{mu}. Choose $w\in W(\frg,\frt)^1$ such that $\mu+2\rho_c \in C_{w}$ and write $\mu+2\rho_c=(x_1,\dots,x_p\mid y_1,\dots,y_q)$.
 	Then
 	\begin{equation}\label{L-0}
 		L(\mu)=\Vert \mu+2\rho_c-\rho^{(w)}\Vert^2-\dfrac{e(\mu)}{2},
 	\end{equation}
 	where
 	\[e(\mu):=\mathrm{card}\{(i,j)\mid 1\le i\le p,1\le j\le q,x_i=y_j\}.\]
Specifically,
\begin{equation}\label{L-1}
 	\begin{split}
 		L(\mu)=\sum_{i=1}^p&(x_i^2-2(n-i+1)x_i)+\sum_{j=1}^q(y_j^2-2(n-j+1)y_j)\\
 		&+\sum_{s=1}^ns^2+2\sum_{i=1}^p\sum_{j=1}^q\min\{x_i,y_j\}-\dfrac{e(\mu)}{2}.
 	\end{split}
 \end{equation}
 \end{lemma}
 \begin{proof}
 	Let $\rho^{(w)}=(z_1,\dots,z_n)$ and
 	\[
 	u=(u_1,\dots,u_n):=\dfrac{1}{\mathrm{card}(W_{\mu+2\rho_c})}\sum_{w\in W_{\mu+2\rho_c}}w\cdot \rho^{(w)}.
 	\]
 	Suppose the number of elements in $W_{\mu+2\rho_c}$ that only permute two indices is $k$. Then
 	\[
 	\mathrm{card}(W_{\mu+2\rho_c})=\sum_{i=0}^k \dbinom{k}{i}=2^k.
 	\]
 	For indices $i$ and $j$ where the components of $\mu+2\rho_c$ are equal, the number of elements in $W_{\mu+2\rho_c}$ that permute $i$ and $j$ is $\sum\limits_{i=0}^{k-1}\dbinom{k-1}{i}=2^{k-1}$. We have
 	\[
 	u_i=u_j=\dfrac{2^{n-1}z_i+2^{n-1}z_j}{2^n}=\dfrac{z_i+z_j}{2}.
 	\]
 	This shows that $u$ is the averaging over $\rho^{(w)}$ the corresponding components of the two indices that have the same components in $\mu+2\rho_c$. We call this operation \emph{the averaging} of $\rho^{(w)}$, and the operation of averaging the corresponding components of a pair of indices is called \emph{an elementary averaging} of $\rho^{(w)}$.
 	
 	Arrange the components of $\mu+2\rho_c=(x_1,\dots,x_p\mid y_1,\dots,y_q)$ in decreasing order as $v_1\ge\cdots\ge v_n$. We claim that one elementary averaging of $\rho^{(w)}$ decreases the value of $\Vert\mu+2\rho_c-\rho^{(w)}\Vert^2$ by $1/2$.
 	
 	Indeed, if $v_t=v_{t+1}=c$, then the corresponding coordinates in $u$ are $n-t+\dfrac{1}{2}$. Thus their contribution after one averaging is
 	\[
 	2\left(c-(n-t+\dfrac{1}{2})\right)^2.
 	\]
 	Their corresponding coordinates in $\rho^{(w)}$ are $m-t+1$ and $m-t$, respectively. Hence their contribution before averaging is
 	\[
 	\big(c-(n-t+1)\big)^2+\big(c-(n-t)\big)^2=2\left(c-(n-t+\dfrac{1}{2})\right)^2+\dfrac{1}{2}.
 	\]
 	This shows that after one elementary averaging of $\rho^{(w)}$, the value of $\Vert\mu+2\rho_c-\rho^{(w)}\Vert^2$ decreases by $1/2$. Notice that $e(\mu)$ is the total number of elementary averagings, we obtain \eqref{L-0}.
 	
 	Note that when $k<l$, we have $\max\{v_k,v_l\}=v_k$. Hence
 	\[
 	\sum_{t=1}^n(n-t+1)v_t=\sum_{t=1}^n v_t+\sum_{1\le k<l\le n}\max\{v_k,v_l\}.
 	\]
 	Since
 	\[
 	\max\{x_i,y_j\}=x_i+y_j-\min\{x_i,y_j\},
 	\]
 	we get
 	\begin{align*}
 		\sum_{1\le t<r\le n}\max\{v_t,v_r\}
 		&=\sum_{i=1}^p(p-i)x_i+\sum_{j=1}^q(q-j)y_j+\sum_{i=1}^p\sum_{j=1}^q\max\{x_i,y_j\}\\
 		&=\sum_{i=1}^p(n-i)x_i+\sum_{j=1}^q(n-j)y_j-\sum_{i=1}^p\sum_{j=1}^q\min\{x_i,y_j\}.
 	\end{align*}
 	Therefore
 	\[
 	\sum_{t=1}^n(n-t+1)v_t=\sum_{i=1}^p(n-i+1)x_i+\sum_{j=1}^q(n-j+1)y_j-\sum_{i=1}^p\sum_{j=1}^q\min\{x_i,y_j\}.
 	\]
 	We obtain \eqref{L-1}.
 \end{proof}
 Set
\begin{equation}\label{N-mu}
 	N(\mu)=\sum_{i=1}^p(x_i^2-2(n-i+1)x_i)+\sum_{j=1}^q(y_j^2-2(n-j+1)y_j)
\end{equation}
and
\begin{equation}\label{M-mu}
M(\mu)=2\sum_{i=1}^p\sum_{j=1}^q\min\{x_i,y_j\}.
\end{equation}
Then \eqref{L-1} becomes
 \begin{equation}\label{L-N-M}
L(\mu)=N(\mu)+M(\mu)-\dfrac{e(\mu)}{2}+\sum_{s=1}^n s^2.
\end{equation}

Sometimes, it will be convenient to study $M(\mu)$ and $-\dfrac{e(\mu)}{2}$ together. For this purpose, define
\begin{equation}\label{function-h}
h(x)=\left\{
 	\begin{array}{ll}
 		-x, & \text{ if } x\le -1, \\
 		-\frac{1}{2}(x-1), & \text{ if } -1<x \leq 0,\\
 		\frac{1}{2}(x+1), &\text{ if } 0<x\leq 1,\\
 		x, &\text{ if } x\ge1.
\end{array}
\right.
\end{equation}
Then $h(d)=|d|+\dfrac{1}{2}\delta_{0, d}$ for any $d\in\bbZ$. Then for each pair of indices $i$ and $j$,
 	\[2\min\{x_i,y_j\}-\dfrac{1}{2}\delta_{x_i, y_j}=(x_i+y_j)-h(x_i-y_j).\]
Therefore,
\begin{equation}\label{M-mu-e-mu}
M(\mu)-\dfrac{e(\mu)}{2}=\sum_{i=1}^{p} \sum_{j=1}^{q} \left( (x_i+y_j)-h(x_i-y_j)\right).
\end{equation}

 \begin{prop} \emph{(Monotonicity of the lambda norm)} \label{lambda-up}
 	Let $\mu$ be $\Delta^+(\frk,\frt)$-dominant weight. Then
 	\[L(\mu)<L(\mu+e_t),\quad 1\le t\le n.\]
 \end{prop}
 \begin{proof}
 	Write $\mu+2\rho_c=(x_1,\dots,x_p\mid y_1,\dots,y_q)$. Choose $w\in\Delta(\frg,\frt)^1$ such that $\mu+2\rho_c\in C_w$. Clearly, the components of $\mu+2\rho_c-\rho^{(w)}$ are nonnegative. Increasing some $x_i$ or $y_j$ by $1$ yields $\mu^\prime$. Then \eqref{L-0} and $|e(\mu)-e(\mu^\prime)|\le 1$ shows that
 \[L(\mu^\prime)-L(\mu)\ge 1-\dfrac{1}{2}|e(\mu^\prime)-e(\mu)|\ge \dfrac{1}{2}.\]
 \end{proof}

Finally, to conclude this section, let us compute the lambda norm of $\texttt{triv}$, which stands for be (the highest weight of) the trivial $K$-type. Choose
 $$
 \rho^{(w)}=(2p-1, 2p-3, \dots, 1\mid \underbrace{p+q, p+q-1, \dots, 2p+1}_{q-p}, 2p, 2p-2, 2p-4, \dots, 2),
 $$
 then $\texttt{triv}+2\rho_c\in \mathcal{C}_w$. Lemma \ref{lambda-1} shows that
 \[P(\texttt{triv}+2\rho_c-\rho^{(w)})=2\rho_c-\lambda_{\texttt{triv}},\]
 where
 	\begin{align*}
 	&\lambda_{\texttt{triv}}=\sum_{w\in W_{2\rho_c}}w\cdot\rho^{(w)}=\\
 	&\bigg(2p-\dfrac{1}{2},2(p-1)-\dfrac{1}{2},\cdots,\dfrac{3}{2}\ \bigg|\  \underbrace{p+q,p+q-1,\dots,2p+1}_{q-p},2p-\dfrac{1}{2},2(p-1)-\dfrac{1}{2},\cdots,\dfrac{3}{2}\bigg).
 \end{align*}
 Then
 \[P(\texttt{triv}+2\rho_c-\rho^{(w)})=\left(\dfrac{1}{2},\dots,\dfrac{1}{2}\ \bigg|\  \underbrace{q-p,q-p-1,\dots,1}_{q-p},\dfrac{1}{2},\dots,\dfrac{1}{2}\right).\]
 We compute that
 \begin{equation}\label{P0}
 	L(\texttt{triv})=\Vert P(\texttt{triv}+2\rho_c-\rho^{(w)})\Vert^2=\sum_{s=1}^{q-p}s^2+\dfrac{p}{2}.
 \end{equation}
In particular,
\begin{equation}\label{pbeta1}
	B_0+L(\texttt{triv})=\sum_{s=1}^{n-1}s^2+(|p-q+1|+1)^2.
\end{equation}

Using the techniques from \cite{DY},  we  find that
 $$
 \min_{m\in\bbZ_{\ge 0}}\Vert \texttt{triv}+m\beta\Vert_{\mathrm{spin}}=\Vert p\beta\Vert_{\mathrm{spin}}.
 $$
Moreover, the spin norm of $p\beta$ is attained at $(q-1, \dots, q-1 \mid p, 0, \dots, 0)\in\Omega_{p, q}$. Note that  $\|p\beta\|_{\rm spin}^2$ is exactly the right hand side of \eqref{pbeta1}. Thus
\begin{equation}\label{B0-explanation}
B_0=B(\texttt{triv}).
\end{equation}

 When $q=p+1$, we point out that
 $$
 \min_{m\in\bbZ_{\ge 0}}\Vert e_1+m\beta\Vert_{\mathrm{spin}}=\Vert e_1+p\beta\Vert_{\mathrm{spin}}.
 $$
Moreover, the spin norm of $e_1+p\beta$ is attained at $(q-1,\dots,q-1\mid p,0,\dots,0)$. A similar computation shows that
 \[B(e_1)=B_0+\dfrac{3}{2}.\]

\section{Convexity of the function $D_w(\mu)$}\label{sec-convexity}
In this section, let $(a\mid b)=(a_1,\dots,a_p\mid b_1, \dots, b_q)$ be $\Delta^+(\frk,\frt)$-dominant weight. Fix any element $\rho_n^{(w)}\in \Omega_{p, q}$ as in \eqref{rho-n-w}.

In this paper, we will frequently analyze the following difference
\begin{equation}\label{D-w-mu}
D_{w}(\mu) :=\Vert \{\mu-\rho_n^{(w)}\} +\rho_c \Vert^2-\|\mu\|_{\rm lambda}^2
\end{equation}
As a preparation, we show that $D_{w}(\mu)$ has certain convexity.

 \begin{prop}\label{Jensen}
 	 We have
 	\begin{itemize} 		
 	\item[(i)] Fix $b$. If $a=\sum_ts_ta^{(t)}$, where $s_t\ge 0$,\ $\sum_ts_t=1$ and $a^{(t)}$ are $p$-dimensional integral vectors with components in decreasing order, then 	
 \[D_w(a\mid b)\le\sum_t s_t D_w(a^{(t)} \mid b ).\]
 \item[(ii)] Fix $a$. If $b=\sum_ts_tb^{(t)}$, where $s_t\ge 0$,\ $\sum_ts_t=1$ and $b^{(t)}$ are $q$-dimensional integral vectors with components in decreasing order, then
     \[D_w(a\mid b)\le\sum_t s_t D_w(a \mid b^{(t)} ).\]
 	\end{itemize}
 \end{prop}
 \begin{proof}
 	 Denote
 	\[(a\mid b)+2\rho_c=(x_1,\dots,x_p\mid y_1,\dots,y_p).\]
 	That is, $x_i=a_i+2P_i$, $y_j=b_j+2Q_j$,  where $P_i=p-i+1$, and $Q_j=q-j+1$.
  Define
 	\begin{align*}
 		\caM_p(a)&=\max_{\sigma\in S_p}\sum_{i=1}^p P_i|a_{\sigma(i)}-k_{\sigma(i)}|,\\
 		\caM_q(b)&=\max_{\sigma\in S_q}\sum_{j=1}^q Q_j|b_{\sigma(j)}-r_{\sigma(j)}|,
 	\end{align*}
 	where $S_p$ and $S_q$ are symmetric groups. By \eqref{Phi}, we have
 	\[\caM_p(a)=\Phi_p(a-K_w),\quad \caM_q(b)=\Phi_q(b-R_w).\]
 	Therefore
 	\begin{equation}\label{a|b1}
 		\Vert (a\mid b)-\rho_n^{(w)}\Vert_\frk^2=\sum_{i=1}^p(a_i-k_i)^2+\sum_{j=1}^q(b_j-r_j)^2+2\caM_p(a)+2\caM_q(b)+\sum_{i=1}^pP_i^2+\sum_{j=1}^qQ_j^2.
 	\end{equation}

 	Thus by \eqref{L-1} and \eqref{M-mu-e-mu},
 	\begin{align*}
 		L(a\mid b)=\sum_{i=1}^p&(x_i^2-2(n-i+1)x_i)+\sum_{j=1}^q(y_j^2-2(n-j+1)y_j)\\
 		&+\sum_{s=1}^ns^2+\sum_{i=1}^p\sum_{j=1}^q[(x_i+y_j)-h(x_i-y_j)],
 	\end{align*}
where $h(x)$ is defined in \eqref{function-h}.
 	Combining with \eqref{a|b1}, we obtain
 	\begin{equation}\label{F}
 		D_{w}(a\mid b)=E(a,b)+2\caM_p(a)+2\caM_q(b)+\sum_{i=1}^p\sum_{j=1}^qh(x_i-y_j),
 	\end{equation}
 	where $E(a,b)$ is linear  in $(a,b)$.
 	
 	For any fixed $\sigma\in S_p$, the function $\sum_i P_i|a_{\sigma(i)}-k_{\sigma(i)}|$ is convex in $a$. Hence
 	\[\sum_i P_i|a_{\sigma(i)}-k_{\sigma(i)}|\le\sum_ts_t\caM_p(a^{(t)}).\]
 	Taking the maximum over $\sigma$ in the above expression gives
 	\[\caM_p(a)\le \sum_ts_t\caM_p(a^{(t)}).\]

Now fix $b$. Since $h$ is a convex function, we have
 	\[h(x_i-y_j) =  h\left(\sum_t (s_t x_i^{(t)}- s_t y_j) \right) \le \sum_ts_th(x_i^{(t)}-y_j).\]
Thus each term in \eqref{F} satisfies Jensen's inequality, the first assertion follows. Note that $h$ is an even function, the second assertion holds as well.
 \end{proof}

\section{Reduction techniques}\label{sec-RT}

For $0\le r<p$, define
\[\zeta_r=(0,\dots,0\mid r,0,\dots,0).\]
Put $\rho_n^{(w_p)}=(q-1, \dots, q-1\mid p, 0, \dots, 0)$.

\begin{lemma}\label{lem1}
	For $0\le r<p$, we have
	\begin{equation}\label{weight-1}
		\Vert\zeta_r+(p-r)\beta-\rho_n^{(w_p)}\Vert_\frk^2-L(\zeta_r)\le B_0
	\end{equation}
\end{lemma}

\begin{proof}
 Note that $p-r\le \caT(\zeta_r)$ by Lemma \ref{m}.

 Firstly, let us consider the case $q\ge p+1$. Arranging $\zeta_r+2\rho_c$ in decreasing order, we obtain the list
     	\[r+2q; \,	\underbrace{2(q-1), 2(q-2), \dots, 2(p+1)}_{q-p-1} ; 2p, 2p, 2(p-1), 2(p-1), \dots, 2, 2.\]
By \eqref{L-0}, we have
\[L(\zeta_r)=(r+q-p)^2+\sum\limits_{s=1}^{q-p-1}s^2+\dfrac{p}{2},\]
	When $q=p+1$, the above sum over $s$ is regarded as $0$.
	
	For $\zeta_r+(p-r)\beta$, we have
	\begin{equation}\label{A}
	\{\zeta_r+(p-r)\beta-\rho_n^{(w_p)}\}=(q-1,\dots,q-1,q-1-p+r\mid 0,\dots,0).
\end{equation}
	Thus 	\[\Vert\zeta_r+(p-r)\beta-\rho_n^{(w_p)}\Vert_\frk^2=\sum_{t=1}^{n-1}t^2+(q-p+r)^2.\]
	Then
	\[\Vert\zeta_r+(p-r)\beta-\rho_n^{(w_p)}\Vert_\frk^2-L(\zeta_r)=\sum_{t=q-p}^{n-1}t^2-\dfrac{p}{2}.\]
	By \eqref{B0}, we obtain
	\[\Vert\zeta_r+(p-r)\beta-\rho_n^{(w_p)}\Vert_\frk^2-L(\zeta_r)=B_0.\]
	
	Now consider the case that $q=p$. It is easy to obtain
	\[L(\zeta_r)=\left\{
	\begin{array}{ll}
		\dfrac{p}{2},& r=0,\\
		r^2+1+\dfrac{p-1}{2}, & r>0.
	\end{array}
	\right.\]
	Then
	\[\{\zeta_r+(p-r)\beta-\rho_n^{(w_p)}\}=(p-1,\dots,p-1,|r-1|\mid 0,\dots,0).\]
	Therefore
	\[\Vert\zeta_r+(p-r)\beta-\rho_n^{(w_p)}\Vert_\frk^2=\sum_{t=1}^{n-1}t^2+(|r-1|+1)^2.\]
	By \eqref{B0},
	\[
	\Vert\zeta_r+(p-r)\beta-\rho_n^{(w_p)}\Vert_\frk^2-L(\zeta_r)=\left\{
	\begin{array}{ll}
		B_0, & r=0,\\
		B_0-\dfrac{9}{2}, & 1\le r<p.
	\end{array}
	\right.
	\]
	This finishes the proof.	
\end{proof}

For $0\le k\le p$, define
\[\eta_k=(q-k,0,\dots,0\mid 0,\dots,0).\]
Put $\rho_n^{(w_0)}=(q, 0, \dots, 0 \mid p-1, \dots, p-1)$. For $1\leq k\leq p-1$, let
\[
\rho_n^{(w_k)}=
	(q, \underbrace{q-1, \dots, q-1}_k, \underbrace{0, \dots, 0}_{p-k-1}\mid p-1, p-k-1, \dots, p-k-1).
\]
Recall that $\rho_n^{(w_p)}=(q-1, \dots, q-1\mid p, 0, \dots, 0)$.
Note that $\rho_n^{(w_k)}\in \Omega_{p, q}$ for $0\leq k\leq p$.
Then we compute that $\{\eta_k + k\beta - \rho_n^{(w_k)}\}$ is equal to
\begin{equation}\label{eta-k}
\begin{cases}
(0, \dots, 0 \mid p-1, \dots, p-1), & k = 0, \\
(\underbrace{q-1, \dots, q-1}_{k}, 0, \dots, 0 \mid p-k-1, \dots, p-k-1), & 1 \le k \le p-1, \\
(q-1, \dots, q-1, 1 \mid 0, \dots, 0), & k = p.
\end{cases}
\end{equation}

\begin{lemma}\label{lemma-etak}
For  $0\le k\le p$, we have
\begin{equation}\label{weight-2}
\Vert\eta_k+k\beta-\rho_n^{(w_k)}\Vert_\frk^2-L(\eta_k)\le\left\{
\begin{array}{ll}
			B_0, &  \mbox{if } q\neq p+1,\\
			B_0+\dfrac{3}{2}, & \mbox{otherwise}.
		\end{array}
		\right.
\end{equation}
\end{lemma}

\begin{proof}
Note that $k\le \caT(\eta_k)$ by Lemma \ref{m}.

Write $\eta_k+2\rho_c$ as $(x_1, \dots, x_p\mid y_1, \dots, y_q)$, then $x_1=q-k+2p$, and the other components are
$$2q, 2(q-1), \dots, 2p$$
and
$$
2(p-1), 2(p-1), 2(p-2), 2(p-2), \dots, 2, 2.
$$
It follows from \eqref{eta-k} that
\[\Vert\eta_k+k\beta-\rho_n^{(w_k)}\Vert_\frk^2=\left\{
\begin{array}{ll}
	\sum\limits_{s=1}^ps^2+\sum\limits_{s=p}^{n-1}s^2,& k=0,\\
	\sum\limits_{s=n-k}^{n-1}s^2+\sum\limits_{s=1}^{p-k}s^2+\sum\limits_{s=p-k}^{n-1-k}s^2, & 1\le k\le p-1,\\
	\sum\limits_{s=q+1}^{n-1}s^2+4+\sum\limits_{s=1}^{q}s^2, &k=p.
\end{array}
\right.\]
For $0\le k\le p-2$, we have	
\begin{equation}\label{diff-1}
\Vert\eta_{k+1}+(k+1)\beta-\rho_n^{(w_{k+1})}\Vert_\frk^2-\Vert\eta_k+k\beta-\rho_n^{(w_k)}\Vert_\frk^2
=1-2(p-k).
\end{equation}
Similarly,
$$
\Vert\eta_p+ p\beta-\rho_n^{(w_p)}\Vert_\frk^2-
\Vert\eta_{p-1}+(p-1)\beta-\rho_n^{(w_{p-1})}\Vert_\frk^2=3.
$$

 For $0\le k\le p-1$, by \eqref{L-1}, we compute that
\[L(\eta_k)-L(\eta_{k+1})=2(p-k)-1+2h_k-\dfrac{l_k-l_{k+1}}{2},\]
where
$$
h_k=\mathrm{card}\{t\mid 1\le t\le q, 2t\ge n+p-k\}
$$
and
$$
l_k=\mathrm{card}\{t\mid 1\le t\le q, 2t=n+p-k\}.
$$
	Since $h_k\ge l_k$ and $h_k\ge l_{k+1}$, we get
\begin{equation}\label{diff-2}
L(\eta_k)-L(\eta_{k+1})\ge 2(p-k)-1.
\end{equation}

Now put $C_k=\Vert\eta_k+k\beta-\rho_n^{(w_k)}\Vert_\frk^2-L(\eta_k)$.   By \eqref{diff-1} and \eqref{diff-2}, we have
\[C_k-C_{k-1}\ge 0\]
for $1\leq k\leq p-1$. Similarly, we have $C_p-C_{p-1}\ge 4$.
This implies that
		\[C_0\le C_1\le\cdots\le C_p.\]

	Now let us prove the lemma.
Firstly, suppose $q=p$. From \eqref{B0}, we get $C_p=B_0$. Hence
\[C_k\le B_0,\quad 0\le k\le p.\]
	
Now assume that $q=p+1$. By \eqref{P0}, we have $L(\texttt{triv})=p/2+1$. Again from \eqref{B0}, we get
$$
\Vert\eta_{p-1}+(p-1)\beta-\rho_n^{(w_{p-1})}\Vert_\frk^2=B_0+L(\texttt{triv}).
$$
Noting that $\eta_{p-1}\gg\texttt{triv}$, Proposition \ref{lambda-up} shows
	\[C_k\le C_{p-1}\leq \Vert\eta_{p-1}+(p-1)\beta-\rho_n^{(w_{p-1})}\Vert_\frk^2-L(\texttt{triv})=B_0,\quad 0\le k<p.\]
	On the other hand, \eqref{L-1} gives $L(\eta_p)=\dfrac{p+5}{2}$. Hence
	\[C_p=\Vert\eta_p+p\beta-\rho_n^{(w_p)}\Vert_\frk^2-L(\eta_p)=B_0+\dfrac{3}{2}.\]
	
Finally, consider the case that $q\ge p+2$. We obtain
	\[\Vert\eta_p+p\beta-\rho_n^{(w_p)}\Vert_\frk^2=\sum_{s=1}^{n-1}s^2+(q-p)^2-((q-p)^2-4)\le \sum_{s=1}^{n-1}s^2+(q-p)^2.\]
	Since $\eta_p\gg\texttt{triv}$, monotonicity of the lambda norm and \eqref{pbeta1} gives
	\[C_k\le C_p=\Vert\eta_p+p\beta-\rho_n^{(w_p)}\Vert_\frk^2- L(\eta_p)< \sum_{s=1}^{n-1}s^2+(q-p)^2-L(\texttt{triv})=B_0,\quad 0\le k\le p.\]
\end{proof}

Now, put
\[P_i=p-i+1,\quad Q_j=q-j+1,\quad 1\le i\le p,1\le j\le q.\]
Let $(a\mid b)=(a_1, \dots, a_p \mid b_1, \dots, b_q)$ be any $\Delta^+(\frk,\frt)$-dominant weight. Set
$$
a^*=(a_1^*, \dots, a_p^*), \quad b^*=(b_1^*, \dots, b_q^*),
$$
where $a_i^*=\min\{a_i, q\}$ and $b_j^*=\min\{b_j, p\}$.
Thus $a^*$ and $b^*$ uniquely determine elements in $\Omega_{p, q}$, which are denoted by
$(a^*\mid R(a))$ and $(K(b)\mid b^*)$, respectively.

\begin{lemma}\label{lem-1}
	Let $(a\mid b)=(a_1, \dots, a_p \mid b_1, \dots, b_q)$ be any $\Delta^+(\frk,\frt)$-dominant weight. Then
		\begin{subequations}
		\begin{equation}\label{1-l}
			\Vert(a\mid b)-(a^*\mid R(a))\Vert_\frk^2-L(a\mid b)\le \Vert(a^*\mid b)-(a^*\mid R(a))\Vert_\frk^2-L(a^*\mid b),
		\end{equation}
		and
		\begin{equation}\label{1-r}
			\Vert(a\mid b)-(K(b)\mid b^*)\Vert_\frk^2-L(a\mid b)\le \Vert(a\mid b^*)-(K(b)\mid b^*)\Vert_\frk^2-L(a\mid b^*).
		\end{equation}
	\end{subequations}
\end{lemma}
\begin{proof}
	We only prove \eqref{1-l}. 	Let $d_i=a_i-a_i^*=\max\{a_i-q,0\}$. Since $a_1\ge\cdots\ge a_p$, we have
$$
d_1\ge\cdots\ge d_p\ge 0.
$$
Hence
\[\{(a\mid b)-(a^*\mid R(a))\}-\{(a^*\mid b)-(a^*\mid R(a))\}=(d_1,\dots,d_p\mid 0,\dots,0).\]
Thus
		\begin{equation}\label{eq1}
		\Vert(a\mid b)-(a^*\mid R(a))\Vert_\frk^2-\Vert(a^*\mid b)-(a^*\mid R(a))\Vert_\frk^2=\sum_{i=1}^p(d_i^2+2P_id_i).
		\end{equation}
Recall that $P_i=p-i+1$ and $Q_j=q-j+1$.
	
	We now compute $L(a\mid b)-L(a^*\mid b)$. Set
	\[x_i=a_i+2P_i,\quad x_i^*=a_i^*+2P_i,\quad y_j=b_j+2Q_j.\]
	In particular, $x_i=x_i^*+d_i$. When $d_i>0$, we have $x_i^*=q+2P_i$, so
	\[N (a\mid b)- N(a^*\mid b)=\sum_{i=1}^{p} (d_i^2 + 2P_i d_i).\]

For any $1\leq i\leq p$ and $1\leq j \leq q$, note that
the difference between
$$
(x_i^* + d_i + y_i) - h(x_i^* + d_i - y_i)
$$
and
$$
(x_i^*  + y_i) - h(x_i^*  - y_i)
$$
is
$$
d_i + h(x_i^*  - y_i) - h(x_i^* + d_i - y_i),
$$
which is non-negative due to the definition of \eqref{function-h}. Therefore,
$$
		M (a\mid b)-M(a^*\mid b) + \frac{1}{2} \left(e(a\mid b)-e(a^*\mid b)\right)\geq 0
$$
by \eqref{M-mu-e-mu}. Now by \eqref{L-N-M}, we have
		\begin{equation}\label{eq2}
		L(a\mid b)-L(a^*\mid b)\ge \sum_{i=1}^p(d_i^2+2P_id_i).
	\end{equation}

Combining \eqref{eq1} and \eqref{eq2}, we obtain \eqref{1-l}.
\end{proof}
Now for $0\le s,t\le q$, put
\[\varphi_s=(\underbrace{p-1, \dots, p-1}_{s},\underbrace{0, \dots, 0}_{q-s}),\quad\phi_t=(\underbrace{2p, \dots, 2p}_{t},\underbrace{0, \dots, 0}_{q-t}).\]
For a $q$-dimensional real vector $v$, let $(K(v)\mid v)$ denote the unique element in $\Omega_{p,q}$ whose last $q$ coordinates are $v$ (if it exists). For instance,
\begin{equation}\label{K-phi-s}
K(\varphi_s)=(q, \underbrace{q-s, \dots, q-s}_{p-1}).
\end{equation}
Define
\begin{equation}\label{F-t-v}
F_t(v):=\Vert(K(v)\mid \phi_t)-(K(v)\mid v)\Vert_\frk^2-L(K(v)\mid \phi_t).
\end{equation}

\begin{lemma}\label{lem-p}
Suppose $p\ge 2$. Fix $\rho_n^{(w)}\in\Omega_{p,q}$ as in \eqref{rho-n-w}. Assume in addition that $k_1=q$.
	If for every $0\le s\le q$ and $0\le t\le q$, we have $F_t(\varphi_s)\le B_0$, then
	\[F_t(R_w)=D_w(K_w\mid \phi_t)\le B_0.\]
\end{lemma}
\begin{proof}	For $1\le s\le q$ and $r_{s}\le l\le p-1$, define the $q$-dimensional vector
	\[\varphi_{s,l}=(\underbrace{p-1,\dots,p-1}_{s-1},l,r_{s+1},\dots,r_q)\]
Recall that $R_w=(r_1,\dots,r_q)$.	
Thus $\varphi_{1,r_1}=R_w$. Note also that
\begin{equation}\label{var}
\varphi_{1, p-1}=\varphi_{2, r_2}, \varphi_{2, p-1}=\varphi_{3, r_3}, \dots,
\varphi_{q-1, p-1}=\varphi_{q, r_q}, \varphi_{q, p-1}=\varphi_{q}.
\end{equation}
For $r_{s}+1\le l\le p-1$, set
\[\sigma_{t,s}(l):=F_t(\varphi_{s,l-1})-F_t(\varphi_{s,l}).\]

For $1\le s\le q$ and $0\le l\le p-1$, define
\[\varphi_{s,l}^0=(\underbrace{p-1,\dots,p-1}_{s-1},l,0,\dots,0).\]
For $1\le l\le p-1$, set
	\[\sigma_{t,s}^0(l):=F_t(\varphi^0_{s,l-1})-F_t(\varphi^0_{s,l}).\]

We will show that $\sigma_{t,s}(l)$ is monotonically decreasing with respect to $l$ and
	\[\sigma_{t,s}(l)=\sigma_{t,s}^0(l).\]
It follows that for any $r_s\le z\le p-1$,
	\[\sum_{l=z+1}^{p-1}\sigma_{t,s}(l)\le \frac{p-1-z}{p-1}\sum_{l=1}^{p-1}\sigma_{t,s}^0(l).\]
	Note that	\[\sum_{l=z+1}^{p-1}\sigma_{t,s}(l)=F_t(\varphi_{s,z})-F_{t}(\varphi_{s,p-1})\]
	and
\[
\sum_{l=1}^{p-1}\sigma^0_{t,s}(l)
=F_t(\varphi_{s,0}^0)-F_{t}(\varphi_{s,p-1}^0)
=F_t(\varphi_{s-1})-F_t(\varphi_s).\]
	Thus
	\begin{align*}
		F_t(R_w)&= F_t(\varphi_{1, r_1})\\
&=F_t(\varphi_{q}) + \sum_{s=1}^{q} (F_t(\varphi_{s, r_s}) - F_t(\varphi_{s, p-1}))\\
&=F_t(\varphi_{q})+\sum_{s=1}^q\sum_{l=r_s+1}^{p-1}\sigma_{t,s}(l)\\
		&\le F_t(\varphi_{q})+\sum_{s=1}^q  \frac{p-1-r_s}{p-1} \sum_{l=1}^{p-1} \sigma^0_{t,s}(l)\\
&= F_t(\varphi_{q})+\sum_{s=1}^q  \frac{p-1-r_s}{p-1} (F_t(\varphi_{s-1})-F_t(\varphi_s))\\
&=\sum_{s=0}^qc_s F_t(\varphi_s),
	\end{align*}
	where the second step uses \eqref{var} and
	\[c_0=1-\dfrac{r_1}{p-1},\quad c_i=\dfrac{r_{i}-r_{i+1}}{p-1}\quad(1\le i\le q-1),\quad
	c_q=\dfrac{r_q}{p-1}\]
	satisfies $c_s\ge 0$ and $\sum_sc_s=1$. Therefore, the lemma holds.

Now let us analyze $\sigma_{t,s}(l)$, which is the sum of
\begin{equation}\label{term-1}
\|(K(\varphi_{s, l-1})\mid\phi_t) - (K(\varphi_{s, l-1})\mid\varphi_{s, l-1})\|_{\frk}^2 - \|(K(\varphi_{s, l})\mid\phi_t) - (K(\varphi_{s, l})\mid\varphi_{s, l})\|_{\frk}^2
\end{equation}
and
\begin{equation}\label{term-2}
L(K(\varphi_{s, l})\mid\phi_t)- L(K(\varphi_{s, l-1})\mid\phi_t).
\end{equation}

Firstly, we calculate \eqref{term-1}. Note that
\[\|(K(\varphi_{s, l-1})\mid\phi_t) - (K(\varphi_{s, l-1})\mid\varphi_{s, l-1})\|_{\frk}^2=\Phi_q(\phi_t-\varphi_{s,{l-1}})\]
and
\[\|(K(\varphi_{s, l})\mid\phi_t) - (K(\varphi_{s, l})\mid\varphi_{s, l})\|_{\frk}^2=\Phi_q(\phi_t-\varphi_{s,l}).\]
 Suppose $s\le t$. Then
\begin{align*}
	\{\phi_t-\varphi_{s,l}\}_q&=(\underbrace{2p-r_t,\dots,2p-r_{s+1}}_{t-s},\underbrace{2p-l,p+1,\dots,p+1}_{s},r_{t+1},\dots,r_q),\\
	\{\phi_t-\varphi_{s,l-1}\}_q&=(\underbrace{2p-r_t,\dots,2p-r_{s+1}}_{t-s},\underbrace{2p-l+1,p+1,\dots,p+1}_{s},r_{t+1},\dots,r_q).
\end{align*}
Suppose $s>t$. We have
\begin{align*}
	\{\phi_t-\varphi_{s,l}\}_q&=(\underbrace{p+1,\dots,p+1}_{t},\underbrace{p-1,\dots,p-1}_{s-t-1},l,r_{s+1},\dots,r_q),\\
	\{\phi_t-\varphi_{s,l-1}\}_q&=(\underbrace{p+1,\dots,p+1}_{t},\underbrace{p-1,\dots,p-1}_{s-t-1},l-1,r_{s+1},\dots,r_q).
\end{align*}
Thus the value of \eqref{term-1} is
\begin{equation}\label{term-1-1}
\Phi_q(\phi_t-\varphi_{s,{l-1}})-\Phi_q(\phi_t-\varphi_{s,l})=\left\{
\begin{array}{ll}
	-2l+2(n+p+s-t)+1, & s\le t,\\
	-2l-2(q-s)-1,& s>t.
\end{array}
\right.
\end{equation}

	Now compute \eqref{term-2}. Let $(K(\varphi_{s,l})\mid \phi_t)+2\rho_c=(x_1,\dots,x_p\mid y_1,\dots,y_q)$. Note that
	\[K(\varphi_{s,l})-K(\varphi_{s,l-1})=(\underbrace{0,\dots,0,-1}_{p-l+1},0,\dots,0),\]
	and $x_{p-l+1}=q-s+2l$, we have
	\begin{equation}\label{term-2-1}
		N(K(\varphi_{s,l})\mid \phi_t)-N(K(\varphi_{s,l-1})\mid \phi_t)=x_{p-l+1}^2-(x_{p-l+1}+1)^2+2(q+l)=-2l+2s-1.
	\end{equation}
On the other hand,
\begin{equation}\label{term-2-2}
	\begin{split}
	M(K(\varphi_{s,l})\mid \phi_t)-M(K(\varphi_{s,l-1})\mid\phi_t)&=\sum_{j=1}^q(\min\{x_{q-l+1},y_j\}-\min\{x_{q-l+1}+1,y_j\})\\
	&=-2\,\mathrm{card}\{j\mid y_j\ge x_{q-l+1}+1\}\\
	&=-2\, \mathrm{card}\{j\mid y_j\ge q-s+2l+1\}.
	\end{split}
\end{equation}
Finally, $e(K(\varphi_{s,l})\mid \phi_t)-e(K(\varphi_{s,l-1})\mid\phi_t)$ is clearly independent of $r_{s+1},\dots,r_q$. Since \eqref{term-1} and \eqref{term-2} remain unchanged when $r_{s+1}=\cdots=r_q=0$, then
\[F_t(\varphi_{s,l-1})-F_t(\varphi_{s,l})=F_t(\varphi_{s,l-1}^0)-F_t(\varphi_{s,l}^0).\]

Note that $|e(K(\varphi_{s,l})\mid \phi_t)-e(K(\varphi_{s,l-1})\mid \phi_t)|\le 1$, then by \eqref{L-N-M}, \eqref{term-1-1}, \eqref{term-2-1}, \eqref{term-2-2},
\[\sigma_{t,s}(l)-\sigma_{t,s}(l+1)\ge 4-2-1=1>0.\]
This shows that $\sigma_{t,s}(l)$ decreases as $l$ increases and we done.
\end{proof}
\begin{lemma}\label{lem-2}
	Fix $\rho_n^{(w)}\in\Omega_{p,q}$ as in \eqref{rho-n-w}. Assume in addition that $k_1=q$.
	Then for $0\le t\le q$,
	\begin{equation}\label{eq3}
		D_w(K_w\mid \phi_t)\le B_0.
	\end{equation}
\end{lemma}
\begin{proof}
	We first prove the case $p=1$. In this case, $\rho_n^{(w)}=(q\mid 0,\dots,0)$. Set
	$$d=M(K_w\mid \phi_t)-M(K_w\mid \phi_{t-1})=\min\{q+2,2Q_t+2\}-\min\{q+2,2Q_t\},$$
	then $2\ge d\ge 0$. By \eqref{L-1},
	\[D_w(K_w\mid \phi_t)-D_w(K_w\mid \phi_{t-1})=4-2d+\dfrac{1}{2}\left(\delta_{q+2,2Q_t+2}-\delta_{q+2,2Q_t}\right)\ge 0.\]
	Hence $D_w(K_w\mid \phi_t)$ is increasing with respect to $t$. We have
	\[D_w(K_w\mid \phi_q)=\dfrac{(q+1)^2}{2}.\]
	When $q=1$, $B_0-D_w(K_w\mid \phi_q)=5/2$. When $q\ge 2$,
	\[B_0-D_w(K_w\mid \phi_q)=\dfrac{3q(q-2)}{2}\ge 0.\]
	This completes the case $p=1$.
	
	Now assume $p\ge 2$. By Lemma \ref{lem-p}, we only need to consider $F_t(\varphi_s)$, where $1\le t,s\le q$. Recall $K(\varphi_s)$ from \eqref{K-phi-s}. Fix $s$. We also write $K(\varphi_s)$ as $(k_1, \dots, k_p)$. Set $x_i=k_i+2P_i$. We compute $F_t(\varphi_s)-F_t(\varphi_{s-1})$, which is the sum of
	\begin{equation}\label{Term-1}
		\|(K(\varphi_s)\mid\phi_t) - (K(\varphi_s)\mid\varphi_s)\|_{\frk}^2 - \|(K(\varphi_s)\mid\phi_{t-1}) - (K(\varphi_s)\mid\varphi_s)\|_{\frk}^2
	\end{equation}
	and
	\begin{equation}\label{Term-2}
		L(K(\varphi_s)\mid\phi_t)- L(K(\varphi_s)\mid\phi_{t-1}).
	\end{equation}

	First compute \eqref{Term-2}. For $1\le i\le p$, set
	\begin{align*}
		d_i&=\big((x_i+2p+2Q_t)-h(x_i-2p-2Q_t)\big)-\big((x_i+2Q_t)-h(x_i-2Q_t)\big)\\
		&= 2p-\big(h(x_i-2p-2Q_t)-h(x_i-2Q_t)\big).
	\end{align*}
	Then
	\[0\le d_i\le 4p.\]
	By \eqref{M-mu-e-mu},
	\begin{equation}\label{M-E}
		\sum_{i=1}^pd_i=\left(M(K(\varphi_s)\mid \phi_t)-\dfrac{e(K(\varphi_s)\mid \phi_t)}{2}\right)-\left(M(K(\varphi_s)\mid \phi_{t-1})-\dfrac{e(K(\varphi_s)\mid \phi_{t-1})}{2}\right).
	\end{equation}
	From \eqref{L-N-M} and \eqref{M-E}, we have
	\[L(K(\varphi_s)\mid \phi_t)-L(K(\varphi_s)\mid \phi_{t-1})=4pQ_t+\sum_{i=1}^pd_i.\]
	
	On the other hand, Note that
	\[\|(K(\varphi_s)\mid\phi_t) - (K(\varphi_s)\mid\varphi_s)\|_{\frk}^2=\Phi_q(\phi_t-\varphi_s).\]
	When $t\le s$,
	\begin{align*}
		\{\phi_t-\varphi_s\}_q&=(\underbrace{p+1,\dots,p+1}_{t},\underbrace{p-1,\dots,p-1}_{s-t},0,\dots,0),\\
		\{\phi_{t-1}-\varphi_s\}_q&=(\underbrace{p+1,\dots,p+1}_{t-1},\underbrace{p-1,\dots,p-1}_{s-t+1},0,\dots,0).
	\end{align*}
	When $t>s$,
	\begin{align*}
		\{\phi_t-\varphi_s\}_q&=(\underbrace{2p,\dots,2p}_{t-s},\underbrace{p+1,\dots,p+1}_{s},0,\dots,0),\\
		\{\phi_{t-1}-\varphi_s\}_q&=(\underbrace{2p,\dots,2p}_{t-s-1},\underbrace{p+1,\dots,p+1}_{s},0,\dots,0).
	\end{align*}
	Therefore the value of \eqref{Term-1} is
	\[\Phi_q(\phi_t-\varphi_s)-\Phi_q(\phi_{t-1}-\varphi_s)=\left\{
	\begin{array}{ll}
		4(p+Q_t), & t\le s,\\
		4p^2+4pQ_t+2s(p-1), & t>s.
	\end{array}
	\right.\]
	Thus
	\[F_t(\varphi_s)-F_{t-1}(\varphi_{s})=\left\{
	\begin{array}{ll}
		4(p+Q_t)-4pQ_t-\sum_{i=1}^pd_i, & t\le s,\\
		4p^2+2s(p-1)-\sum_{i=1}^pd_i, & t>s.
	\end{array}
	\right.\]
	Since $p\ge 2$, the term $4(p+Q_t)-4pQ_t$ is non-positive for $Q_t\ge 2$. When $Q_t=1$, note that $x_1=q+2p\ge 2p+2=2p+Q_t$, so $d_1\ge 4p-1/2$. Hence $F_t(\varphi_s)-F_{t-1}(\varphi_{s})\le 0$ whenever $t\le s$. On the other hand, since $d_i\le 4p$, we have $\sum_{i=1}^pd_i\le 4p^2$, so $F_t(\varphi_s)-F_{t-1}(\varphi_{s})>0$ when $t>s$. This shows
	\[F_t(\varphi_s)\le \max\{F_0(\varphi_s), F_q(\varphi_{s})\}.\]
	
	Consider the case $t=0$. By Proposition \ref{lambda-up}, $L(K(\varphi_s)\mid 0)\ge L(K(\varphi_q)\mid 0)$. Since the coordinates of $\varphi_s$ are all at most $p-1$, it is clear that
	\[\Phi_q(\varphi_s)\le \Phi_q(\varphi_q).\]
	Hence
	\[F_0(\varphi_s) \le \Phi_p(0)+\Phi_q(\varphi_q)-L(K(\varphi_q)\mid 0)=F_0(\varphi_q).\]
Note that $F_0(\varphi_q)=C_0$,	where $C_0$ is in defined in the proof of Lemma \ref{lemma-etak}. It has been shown there that $C_0\leq B_0$. Thus $F_0(\varphi_q)\le B_0$. Now it remains to show $F_q(\varphi_s)\le B_0$. This will be done in the next lemma.
\end{proof}

\begin{lemma}
Suppose $p\ge 2$. We have $F_q(\varphi_s)\le B_0$ for $0\leq s\leq q$.
\end{lemma}
\begin{proof}
Set $y_j=2p+2Q_j$, then $$(K(\varphi_s)\mid\phi_q)+2\rho_c=(x_1,\dots,x_p\mid y_1,\dots,y_q).$$
	
	Firstly, assume that $q>p$. By \eqref{L-N-M} and $e(K(\varphi_s)\mid\phi_q)\le p$,
	\begin{equation}\label{eq-1}
		\begin{aligned}
			L(K(\varphi_s)\mid\phi_q)\ge N(K(\varphi_s)\mid \phi_q)+M(K(\varphi_s)\mid\phi_q)+\sum_{u=1}^nu^2-\dfrac{p}{2}.
		\end{aligned}
	\end{equation}
		Note that
	\[\min\{k_i+2P_i,2p+2Q_j\}\ge 2\min\{P_i,Q_j\}+\min\{k_i,2p\}.\]
	Then
	\begin{equation}\label{q-0}
		M(K(\varphi_s)\mid\phi_q)\ge4\sum_{i=1}^p\sum_{j=1}^q\min\{i,j\}+2q\big(\min\{q,2p\}+(p-1)\min\{q-s,2p\}\big).
	\end{equation}
	A straightforward computation gives
	\begin{equation}\label{q-1}
		\sum_{i=1}^p\sum_{j=1}^q\min\{i,j\}=\dfrac{p(p+1)(3q-p+1)}{6}
	\end{equation}
	and
	\begin{equation}\label{q-2}
		 N(K(\varphi_s)\mid \phi_q)=\sum_{i=1}^p\big(x_i^2-2(q+P_i)x_i\big)=-q(q+2p)+(p-1)(s^2-ps-q^2-pq).
	\end{equation}
	
		On the other hand,
	\begin{equation}\label{q-3}
		\Vert(K(\varphi_s)\mid\phi_q)-(K(\varphi_s)\mid \varphi_s)\Vert_\frk^2=\sum_{i=1}^pi^2+\sum_{j=s+1}^q(2p+j)^2+\sum_{j=1}^{s}(p+1+j)^2.
	\end{equation}

\begin{itemize}
\item [$\bullet$]	If $q\le 2p$, then $\min\{q,2p\}=q$ and $\min\{q-s,2p\}=q-s$. Combining \eqref{B0}, \eqref{eq-1}, \eqref{q-0}, \eqref{q-1}, \eqref{q-2} and \eqref{q-3}, we obtain
	\[\dfrac{B_0-F_q(\varphi_s)}{2}\ge (p-1)s^2+(q-p-1)\big(q+(p-1)(q-s)\big)\ge 0.\]
\item [$\bullet$]	If  $q>2p$, then $\min\{q,2p\}=2p$. Similarly, we get
	\[\dfrac{B_0-F_q(\varphi_s)}{2}\ge (p-1)\big(s(s+1)+q\min\{q-s,2p\}-p(q-s)\big)\ge 0.\]
\end{itemize}
Thus $F_q(\varphi_s)\le B_0$ whenever $q>p$. This finishes the case $q>p$.

Secondly, assume that $q=p$. We \emph{claim} that $F_p(\varphi_s)$ decreases as $s$ increases. Indeed, from
	\begin{align*}
	\{\phi_p-\varphi_{s}\}_p&=(\underbrace{2p,\dots,2p}_{p-s},p+1,\dots,p+1),\\
		\{\phi_p-\varphi_{s+1}\}_p&=(\underbrace{2p,\dots,2p}_{p-s-1},p+1,\dots,p+1),
	\end{align*}
	we have
	\begin{align*}
		\Phi_q(\phi_p-\varphi_{s})-\Phi_q(\phi_p-\varphi_{s+1})&=(2p+s+1)^2-(p+s+2)^2\\
		&=(p-1)(3p+2s+3)\\
		&\ge(p-1)(3p+3).
	\end{align*}
	On the other hand, since
	\[K(\varphi_s)-K(\varphi_{s+1})=(0,1,\dots,1),\]
	then
	\[N(K(\varphi_s)\mid\phi_p)-N(K(\varphi_{s+1})\mid\phi_p)=\sum_{i=1}^{p-1}(2i-2s-1)=(p-1)(p-1-2s)\le (p-1)^2.\]
	Note that
	\[M(K(\varphi_s)\mid\phi_p)-M(K(\varphi_{s+1})\mid\phi_p)=2\sum_{i=2}^p\sum_{j=1}^p(\min\{x_i,y_j\}-\min\{x_i-1,y_j\})\le 2p(p-1)\]
	and $|e(K(\varphi_s)\mid\phi_p)-e(K(\varphi_{s+1})\mid\phi_p)|\le p-1$, then by \eqref{L-N-M},
	\[F_p(\varphi_s)-F_p(\varphi_{s+1})\ge (p-1)(3p+3)-(p-1)^2-2p(p-1)-\dfrac{p-1}{2}=\dfrac{7}{2}(p-1)>0.\]
	This proves the claim. Hence it suffices to show $F_p(\varphi_0)\le B_0$.
	
	Note that $(K(\varphi_0)\mid \phi_p)=(p,\dots,p\mid 2p,\dots,2p)$. If $i\le j$, then
$$
\min\{p+2i,2p+2j\}=p+2i.
$$
If $i>j$, then
	\[\min\{p+2i,2p+2j\}=p+2j+\min\{2(i-j),p\}\ge p+2j+(i-j+1)\]
since $p\ge i-j+1$. 	 By \eqref{q-1},
	\[\sum_{i=1}^p\sum_{j=1}^p\min\{i,j\}=\dfrac{p(p+1)(2p+1)}{6}.\]
	For $i>j$, we have
	\[\sum_{p\ge i>j\ge 1}(i-j+1)=\dfrac{p(p-1)(p+4)}{6}.\]
	Therefore,
	\begin{align*}
		\dfrac{1}{2} M(K(\varphi_0)\mid \phi_p)&\ge p^3+2\sum_{i=1}^p\sum_{j=1}^p\min\{i,j\}+\sum_{i>j}(i-j+1)\\
		&=p^3+\dfrac{5p^3+9p^2-2p}{6}.
	\end{align*}
	Moreover,
	\[N(K(\varphi_0)\mid\phi_p)=\sum_{i=1}^p\big((p+2i)^2-2(p+i)(p+2i)\big)=-2p^3-p^2.\]
	Since $e(K(\varphi_0)\mid \phi_p)\le p$, then \eqref{L-N-M} yields
	\begin{align*}
		L(K(\varphi_0)\mid \phi_p)&= N(K(\varphi_0)\mid\phi_p)+M(K(\varphi_0)\mid\phi_p)+\sum_{u=1}^{2p}u^2-\dfrac{e(K(\varphi_0)\mid \phi_p)}{2}\\
		&\ge -p^2+\dfrac{5p^3+9p^2-2p}{3}+\dfrac{2p(2p+1)(4p+1)}{6}-\dfrac{p}{2}\\
		&=\dfrac{26p^3+24p^2-5p}{6}.
	\end{align*}
	On the other hand,
	\[\Vert(K(\varphi_0)\mid\phi_t)-(K(\varphi_0)\mid\varphi_p)\Vert_\frk^2=\sum_{i=1}^pi^2+\sum_{i=1}^p(2p+i)^2=\dfrac{20p^3+9p^2+p}{3}.\]
	So \eqref{B0} gives
	\begin{align*}
		B_0-F_p(\varphi_0)&\ge \dfrac{16p^3-12p^2-p+24}{6}-\dfrac{20p^3+9p^2+p}{3}+\dfrac{26p^3+24p^2-5p}{6}\\
		&=\dfrac{(p-2)(p-3)(p+2)}{3}\ge 0.
	\end{align*}
	This proves the lemma.
\end{proof}

\section{Proof of Theorem \ref{thm-main}: the u-small case}\label{sec-u-small}
\begin{prop}
	Take $\mu$ as in \eqref{mu}. Assume in addition that $\mu$ is in $\caU_{p,q}$ and $a_1<q$, $b_1<p$. If $q-a_1>p-b_1$, then $B(\mu)\le B_0$.
\end{prop}
\begin{proof}
	 By lemma \ref{m},  $p-b_1\le\caT(\mu)$. Take $R_w=(p,b_2,\dots,b_q)$. It uniquely determines an element $\rho_n^{(w)}=(K_w\mid R_w)\in \Omega_{p, q}$. Let $K_w=(k_1, \dots, k_p)$.
	
Note that
	\[\mu+(p-b_1)\beta=(a_1+p-b_1, a_2, \dots, a_p \mid p, b_2, \dots, b_q).\]
	Since $p>b_2$, we have $k_1=q-1$. Note that $a_1+(p-b_1)<q$ by our assumption, then
	\[|a_1+p-b_1-k_1|=(q-1)-(a_1+p-b_1)\le q-1-(p-b_1).\]
	Clearly, $0\le a_i\le q-1$ and $0\le k_i\le q-1$ for all $2\le i\le p$. This shows that $|a_i-k_i|\le q-1$ for $2\le i\le p$. By \eqref{A},
	\[\{\zeta_{b_1}+(p-b_1)\beta-\rho_n^{(w_{b_1})}\}=(q-1, \dots, q-1, q-1-(p-b_1)\mid 0, \dots, 0),\]
which weakly dominates $\{\mu+(p-b_1)\beta-\rho_n^{(w)}\}$. Therefore  by \eqref{weakly-dominate},	\[\Vert\mu+(p-b_1)\beta-\rho_n^{(w)}\Vert_\frk^2\le\Vert\zeta_{b_1}+(p-b_1)\beta-\rho_n^{(w_{b_1})}\Vert_\frk^2.\]
	Nota that $\mu\gg\zeta_{b_1}$. Then Proposition \ref{lambda-up} implies that $L(\mu)\ge L(\zeta_{b_1})$. Therefore,
	\[B(\mu)\le \Vert\zeta_{b_1}+(p-b_1)\beta-\rho_n^{(w_{b_1})}\Vert_\frk^2-L(\zeta_{b_1})\le B_0,\]
where the last step uses \eqref{weight-1}.
\end{proof}

\begin{lemma}\label{lem2}
Take $\mu$ as in \eqref{mu}. Assume in addition that $a_1=q>a_2$ and $b_1\leq p-1$. Then there exists $\rho_n^{(w)}\in \Omega_{p, q}$ such that
	\[\{\mu-\rho_n^{(w)}\}=(m_1, \dots, m_p \mid n_1, \dots, n_q)\]
satisfies
	\begin{itemize}
		\item[(i)] For any $1\le f\le p$,\ $\sum_{t=1}^f m_t\le (q-1)\min\{f, b_1\}$.
		\item[(ii)] For any $1\le g\le q$,\ $n_g\le p-1-b_1$.
	\end{itemize}
\end{lemma}
\begin{proof}
Note that there exists a unique element
	\[(a_1,a_2,\dots,a_p\mid y_1,\dots,y_q)\in\Omega_{p,q}.\]
Note that $y_1=p-1$ since $a_1=q>a_2$.

	For each $1\le t\le q$, define
\begin{equation}\label{r-t}
r_t=
\begin{cases}
b_t-(p-1-b_1),  & \text{ if } y_t<b_t-(p-1-b_1),\\
y_t,  & \text{ if } b_t-(p-1-b_1)\leq y_t \leq b_t+(p-1-b_1),\\
b_t+(p-1-b_1),  & \text{ if } y_t > b_t+(p-1-b_1).
\end{cases}
\end{equation}
	Then $r_1=p-1$ since $y_1=p-1$. Moreover, since
$$
y_1\ge\cdots\ge y_q, \quad b_1\ge\cdots\ge b_q,
$$
we have $r_1\ge\cdots\ge r_q\ge 0$. Thus $R_w=(r_1,\dots,r_q)$ determines a unique element
	\[\rho_n^{(w)}=(K_w\mid R_w)=(k_1,\dots,k_p\mid r_1,\dots,r_q)\in \Omega_{p,q}.\]
It is obvious from \eqref{r-t} that $|b_t-r_t|\le p-1-b_1$, which is (ii).

Note that $|y_t-r_t|\le b_1$.	
Indeed, by \eqref{r-t}, whenever $r_t\neq y_t$, there are two cases:
		\begin{itemize}
		\item[$\bullet$] $r_t<y_t$, then $y_t-r_t\le p-1-b_t-(p-1-b_1)=b_1-b_t\le b_1$.
		\item[$\bullet$] $r_t>y_t$, then $r_t-y_t\le r_t=b_t-(p-1-b_1)=b_1+b_t-(p-1)\le b_1$.
	\end{itemize}
	Since $y_1=r_1=p-1$, we have
		\[\sum_{t=1}^q|y_t-r_t|\le (q-1)b_1.\]
	By Lemma \ref{omegapq}, we have
	\[\sum_{t=1}^p|a_t-k_t|=\sum_{t=1}^q|y_t-r_t|\le (q-1)b_1.\]
	Since $a_1=k_1=q$, the equality $r_1=p-1$ implies $k_2\le p-1$. Also $a_2\le p-1$, so for any $1\le t\le p$, we have $m_t\le p-1$. Therefore
	\[\sum_{t=1}^fm_t\le\min\{(q-1)b_1,(q-1)f\}=(q-1)\min\{f, b_1\},\]
	which is exactly (i).
\end{proof}
\begin{prop}
Take $\mu$ as in \eqref{mu}.	Assume that $\mu$ is in $\caU_{p,q}$ and $a_1<q$, $b_1<p$ and that $q-a_1\le p-b_1$. If $q=p+1$, then $B(\mu)\le B_0+3/2$. Otherwise, $B(\mu)\le B_0$.
\end{prop}
\begin{proof}
	Take $s=b_1+(q-a_1)$, then $s\le p$. By Lemma \ref{m}, we have $q-a_1\le \caT(\mu)$. At this point,
	\[\mu+(q-a_1)\beta=(q,a_2,\dots,a_p\mid s,b_2,\dots,b_q).\]
	In particular, $a_2<q$.
	
	Firstly, suppose $s\le p-1$. Using Lemma \ref{lem2}, there exists $\rho_n^{(w)}\in\Omega_{p,q}$ such that
\[\{\mu+(q-a_1)\beta-\rho_n^{(w)}\}=(m_1,\dots,m_p\mid n_1,\dots,n_q)\]
	satisfies
	\begin{itemize}
		\item[(a)] For any $1\le f\le p$,\ $\sum_{t=1}^f m_t\le (q-1)\min\{f,s\}$.
		\item[(b)] For any$1\le g\le q$,\ $n_g\le p-1-s$.
	\end{itemize}
	
	Recall the proof of Lemma \ref{lemma-etak}, we have
	\[\{\eta_0-\rho_n^{(w_0)}\}=(0,\dots,0\mid p-1-s,\dots,p-1-s)\]
	and
		\[\{\eta_s+s\beta-\rho_n^{(w_s)}\}=
		(\underbrace{q-1,\dots,q-1}_{s},0,\dots,0\mid p-1-s,\dots,p-1-s),  \quad 1\le s\le p-1.\]
	Hence Lemma \ref{lem-Phi} shows that
	\[\Vert \mu+(q-a_1)\beta-\rho_n^{(w)}\Vert_\frk^2\le \Vert\eta_s+s\beta-\rho_n^{(w_s)}\Vert_\frk^2.\]
	
	Now suppose $s=p$, instead take
	\[R_w=(q,b_2,\dots,b_p),\quad \rho_n^{(w)}=(k_1,\dots,k_p\mid R_w)\in\Omega_{p,q}.\]
	Then
	\[|\mu+(q-a_1)\beta-\rho_n^{(w)}|=(1,|a_2-k_2|,\dots,|a_p-k_p|\mid 0,\dots,0).\]
	Since $q>b_2$, we have $k_1=q-1$, and consequently for any $i\ge 2$, we have $|a_i-k_i|\le q-1$.
	
	On the other hand,
	\[\{\eta_p+p\beta-\rho_n^{(w_p)}\}=(q-1,\dots,q-1,1\mid 0,\dots,0).\]
	Using Lemma \ref{lem-Phi} again, we obtain
		\[\Vert \mu+(q-a_1)\beta-\rho_n^{(w)}\Vert_\frk^2\le \Vert\eta_s+s\beta-\rho_n^{(w_s)}\Vert_\frk^2.\]
	
	Finally, since $a_1=q-(q-a_1)\ge q-s$, we have $\mu\gg\eta_s$. By Proposition \ref{lambda-up}, $L(\mu)\ge L(\eta_s)$. Therefore,
		\[B(\mu)\le\Vert \mu+(q-a_1)\beta-\rho_n^{(w)}\Vert_\frk^2-L(\mu)\le \Vert\eta_s+s\beta-\rho_n^{(w_s)}\Vert_\frk^2-L(\eta_s)=C_s,\]
where $C_s$ is defined in Lemma \ref{lem1}. By the same lemma, $C_s=B_0+3/2$ only when $q=p+1$ and $s=p$, and in all other cases $C_s\le B_0$. The proposition is now proven.
\end{proof}
\begin{prop}\label{a_1>q}
Take $\mu$ as in \eqref{mu}.	Assume that	 $\mu$ is in $\caU_{p,q}$. If $a_1\ge q$, then $B(\mu)\le B_0$.
\end{prop}
\begin{proof}
	Set $\rho_n^{(w)}=(a^*\mid R(a)):=(K_w\mid R_w)$. Then we have
	\[B(\mu)\le D_w(a\mid b)\le D_w(K_w\mid b),\]
where the second step uses Lemma \ref{lem-1}.
	
	By Lemma \ref{m}, we have $2p\ge b_1\ge\cdots\ge b_q\ge 0$. Let
\[s_0=\dfrac{2p-b_1}{2p},\quad s_t=\dfrac{b_t-b_{t-1}}{2p}\quad (1\le t<q),\quad s_q=\dfrac{b_q}{2p}.\]
Therefore,	
	\[D_w(K_w\mid b)\le\sum_{t=0}^qs_tD_w(K_w\mid \phi_t)\le B_0,\]
where  the first step uses Proposition \ref{Jensen}, while the second step uses Lemma \ref{lem-2}.  This finishes the proof.
\end{proof}
\begin{prop}
	Take $\mu$ as in \eqref{mu}.	Assume that	 $\mu$ is in $\caU_{p,q}$. If $b_1\ge q$, then $B(\mu)\le B_0$.
\end{prop}
\begin{proof}
	By Proposition \ref{a_1>q}, we may assume $a_1\le q-1$. Set $\rho_n^{(w)}=(K(a)\mid b^*):=(K_w\mid R_w)$. Then Lemma \ref{lem-1} gives
	\[D_w(a\mid b)\le D_w(a\mid R_w).\]
	Since $r_1=p$, we have $k_1\le q-1$. Thus we have $|a_i-k_i|\le q-1$ for $1\le i\le p$.
	
	Recall that
		\[\zeta_p=(0,\dots,0\mid p,0,\dots,0),\quad \rho_n^{(w_p)}=(q-1,\dots,q-1\mid p,0,\dots,0).\]
	Clearly, $\{\zeta_p-\rho_n^{(w_p)}\}\gg\{(a\mid R_w)-\rho_n^{(w)}\}$. Then
	\[\Vert(a\mid R_w)-\rho_n^{(w)}\Vert_\frk^2\le \Vert \zeta_p-\rho_n^{(w_p)}\Vert_\frk^2.\]
	Moreover, $(a\mid R_w)\gg\zeta_p$. Therefore, Proposition \ref{lambda-up} implies $L(a\mid R_w)\ge L(\zeta_p)$. Then Lemma \ref{lem1} shows that
	\[D_w(a\mid R_w)=\Vert(a\mid R_w)-\rho_n^{(w)}\Vert_\frk^2-L(a\mid R_w)\le \Vert\zeta_p-\rho_n^{(w_p)}\Vert_\frk^2-L(\zeta_p)=B_0.\]
\end{proof}

Combining all the propositions above, we obtain that the Theorem \ref{thm-main} holds for all $\mu\in\caU_{p, q}$.

\section{Proof of Theorem \ref{thm-main}: the u-large case}\label{sec-u-large}
Firstly, let us prepare a few lemmas.
Let $u=(u_1,\dots,u_r)$ and $v=(v_1,\dots,v_r)$ be consisting of  nonnegative real numbers. Set
\[|u-v|=(|u_1-v_1|,\dots,|u_r-v_r|).\]
Arrange $u_1,\dots,u_r$ in decreasing order as $u_{(1)},\dots,u_{(r)}$. Define
\[\caF_r(u)=\sum_{i=1}^r(r-i+1)u_{(i)}\]
and
\[\caD_r(u,v)=\caF_r(u)+\caF_r(v)-\caF_r(|u-v|).\]

\begin{lemma}
Let $u=(u_1,\dots,u_r)$ and $v=(v_1,\dots,v_r)$ be consisting of nonnegative real numbers arranged in decreasing order. Then
	 \begin{equation}\label{D_r-1}
		D_r(u,v)=2\sum_{i=1}^r\min\{u_i,v_i\}+\sum_{i<j}(\min\{u_i-u_j,v_i-v_j\}+\min\{u_i+u_j,v_i+v_j\}).
	\end{equation}
\end{lemma}
\begin{proof}
	Note that
	\[\caF_r(u)=\sum_{i=1}^ru_i+\sum_{i<j}\max\{u_i,u_j\}\]
	 and for real numbers $z_1,z_2$,
	 \[2\min\{z_1,z_2\}=z_1+z_2-|z_1-z_2|,\quad 2\max\{|z_1|,|z_2|\}=|z_1+z_2|+|z_1-z_2|,\]
	  we have
	  \begin{align*}
	  	\caD_r(u,v)&=2\sum_{i=1}^r\min\{u_i,v_i\}+\sum_{i<j}(u_i+v_i-\max\{|u_i-v_i|,|u_j-v_j|\})\\
	  	&=2\sum_{i=1}^r\min\{u_i,v_i\}+\dfrac{1}{2}\sum_{i<j}(2u_i+2v_i-|(u_i-u_j)-(v_i-v_j)|-|(u_i+u_j)-(v_i+v_j)|)\\
	  	&=2\sum_{i=1}^r\min\{u_i,v_i\}+\sum_{i<j}(\min\{u_i-u_j,v_i-v_j\}+\min\{u_i+u_j,v_i+v_j\}).
	  \end{align*}
\end{proof}

\begin{lemma}
		Let $u=(u_1,\dots,u_r)$ and $v=(v_1,\dots,v_r)$ be consisting of nonnegative real numbers arranged in decreasing order. Let $u^\prime=(u_2,\dots,u_p)$,\ $v^\prime=(v_2,\dots,v_q)$, then
			\begin{equation}\label{D_r-2}
			\caD_r(u,v)-\caD_{r-1}(u^\prime,v^\prime)\ge 2\min\{v_1,\sum_{i=1}^ru
			_i\}.
		\end{equation}
\end{lemma}
\begin{proof}
	From \eqref{D_r-1}, we have
	\[\caD_r(u,v)-\caD_{r-1}(u^\prime,v^\prime)=2\min\{u_1,v_1\}+\sum_{j=2}^r(\min\{u_1-u_j,v_1-v_j\}+\min\{u_1+u_j,v_1+v_j\}).\]
	If $u_1\ge v_1$, then we already have $\caD_r(u,v)-\caD_{r-1}(u^\prime,v^\prime)\ge 2v_1$. If $u_1<v_1$, set
$$
d_j=\min\{u_j,v_1-u_1\}
$$
Then $\min\{u_1-u_j,v_1-v_j\}+\min\{u_1+u_j,v_1+v_j\}$ is one of the following numbers:
	 \[2u_1,\quad 2v_1,\quad u_1-u_j+v_1+v_j,\quad u_1+u_j+v_1-v_j.\]
	Clearly, $2u_1$, $2v_1$, and $u_1+u_j+v_1-v_j$ are all at least $2d_j$. If $u_j\le v_1-u_1$, then
$$
u_1-u_j+v_1+v_j\geq 2u_1+v_j  \geq 2u_j.
$$
 Otherwise,
	 \[u_1-u_j+v_1+v_j>2(v_1-u_j)+v_j>2(v_1-u_1).\]
	Therefore we always have $\min\{u_1-u_j,v_1-v_j\}+\min\{u_1+u_j,v_1+v_j\}\ge 2d_j$. Hence
	 \[\caD_r(u,v)-\caD_{r-1}(u^\prime,v^\prime)\ge 2u_1+2\sum_{j=2}^rd_j\ge 2\min\{v_1,\sum_{i=1}^ru_1\}.\]
	This proves the lemma.
\end{proof}

\begin{lemma}\label{lemma-Phi-Dr}
		Let $u=(u_1,\dots,u_r)$ and $v=(v_1,\dots,v_r)$ be consisting of nonnegative real numbers arranged in decreasing order. Then
		\begin{equation}\label{Phi-Dr}
			\Phi_r(|u-v|)-\sum_{i=1}^r(u_i-v_i+r-i+1)^2=4\caF_r(v)-2\caD_r(u,v).
		\end{equation}
\end{lemma}
\begin{proof}
	Let $|u-v|_{(1)},\dots,|u-v|_{(r)}$ be the decreasing rearrangement of the components of $|u-v|$. From \eqref{Phi}, we have
	\begin{align*}
\Phi_r(|u-v|)-\sum_{i=1}^r(u_i-v_i+r-i+1)^2 &= \sum_{i=1}^r(|u-v|_{(i)}+r-i+1)^2 - \sum_{i=1}^r(u_i-v_i+r-i+1)^2 \\
&= 2\sum_{i=1}^r(r-i+1)\big(|u-v|_{(i)}-u_i+v_i\big)\\
&=2\big(\caF_r(|u-v|))-\caF_r(u)+\caF_r(v)\big).
\end{align*}
	On the other hand, by the definition of $\caD_r$,
	\[4\caF_r(v)-2\caD_r(u,v)=2\big(\caF_r(|u-v|))-\caF_r(u)+\caF_r(v)\big).\]
	This proves the lemma.
\end{proof}

Now we put $\mu=(a_1, \dots, a_p \mid b_1, \dots, b_q)$ as the highest weight of a $K$-type. Let $w$ be in $W(\frg,\frt)^1$ and set $\rho_n^{(w)}$ as in \eqref{rho-n-w}.  Since $\rho_n^{(w)}$ is actually the half-sum of the positive roots in $w\Delta^+(\frp,\frt)$, and noncompact roots of the form $e_i+e_{p+j}$ are necessarily positive roots, the remaining noncompact positive roots must be of the form
\[e_i-e_{p+j}\quad\mbox{or}\quad e_{p+j}-e_i,\quad 1\le i\le p,1\le j\le q.\]
If $e_i-e_{p+j}$ is a noncompact positive root, the contribution of the index pair $(i,j)$ to $\rho_n^{(w)}$ is
\[\dfrac{1}{2}\big((e_i+e_{p+j})+(e_i-e_{p+j})\big)=e_i.\]
On the other hand, if $e_{p+j}-e_i$ is a noncompact positive root, then the contribution of the index pair $(i,j)$ to $\rho_n^{(w)}$ is
\[\dfrac{1}{2}\big((e_i+e_{p+j})+(e_{p+j}-e_{i})\big)=e_{p+j}.\]
Therefore, we have
\begin{equation}\label{rho_nw}
	k_i=\mathrm{card}\{j:e_i-e_{p+j}\in w\Delta^+(\frg, \frt)\},\quad r_j=\mathrm{card}\{i:e_{p+j}-e_i\in w\Delta^+(\frg, \frt)\}.
\end{equation}
In particular, since $e_{p+l}-e_{p+j}$ $(1\le l<j\le p)$ are compact positive roots, this shows that whenever $e_{p+j}-e_i\in w\Delta^+(\frg, \frt)$, we also have $e_{p+l}-e_i\in w\Delta^+(\frg, \frt)$. Similarly, we obtain
  \begin{equation}\label{rho_nw-2}
  	\begin{split}
	\{i:e_i-e_{p+j}\in w\Delta^+(\frg, \frt)\}&=\{1,\dots,q-r_j\},\\ \{j:e_{p+j}-e_i\in w\Delta^+(\frg, \frt)\}&=\{1,\dots,q-k_i\}.
	\end{split}
\end{equation}
\begin{lemma}\label{lem-large}
Take $\mu$ as in \eqref{mu}.	Assume that  $\mu$ is in $\caU_{p,q}^c$. Write
	$\rho_n^{(w)}$ as in \eqref{rho-n-w} for $w\in W(\frg,\frt)^1$.
	If $\mu+2\rho_c\in \caC_w$, then
	\begin{equation}\label{D_r-3}
		\caD_p(a,K_w)+\caD_q(b,R_w)\ge 2pq.
	\end{equation}
\end{lemma}
\begin{proof}
	Consider the pair of nonnegative integers $(p,q)$, where we no longer require that $p\le q$. We prove the lemma by induction on $p+q$. The case $p=0$ or $q=0$ is trivial. Thus it suffices to assume that $p>0$ and $q>0$. Recalling the action of $w$ in Section \ref{sec-spin-norm}, we find that  $w\xi_1$ has only two choices: $e_1$ or $e_{p+1}$.

Firstly, let us consider the case that  $w\xi_1=e_1$. Then $e_1-e_{p+j}\in w\Delta^+(\frg, \frt)$ for all $1\le j\le q$. From \eqref{rho_nw}, we have $k_1=q$. We move on according to the following three cases:
\begin{itemize}
\item[(i)]  there exists $2\leq i\leq p$ such that $a_i>2 k_i$;
\item[(ii)] there exists $1\leq j\leq q$ such that $b_j>2 r_j$;
\item[(iii)] $a_i\leq 2 k_i$ for each $2\leq i\leq p$ and $b_j\leq 2 r_j$ for each $1\leq j\leq q$.
\end{itemize}

We \emph{claim} that
		\begin{equation}\label{sum-a}
		\sum_{t=1}^p a_t\ge q
	\end{equation}
holds in cases (i) and (ii).
	Indeed, to prove \eqref{sum-a}, it suffices to assume that $a_1<q$.

In case (i), there exists $i\ge 2$ such that $a_i>2k_i$. Set
	\[h=\mathrm{card}\{j:e_{p+j}-e_i\in w\Delta^+(\frg, \frt)\}=q-k_i>q-a_i\ge 1.\]
	Then $a_i\ge 2(q-h)+1$. From \eqref{rho_nw-2}, we have $e_{p+h}-e_i\in w\Delta^+(\frg, \frt)$. Thus
	\[b_{h}+2(q-h+1)\ge a_i+2(p-i+1).	\]
	Furthermore, since $e_1-e_{p+1}\in w\Delta^+(\frg, \frt)$, we have
		\[a_1\ge b_1+2(q-p)\ge b_{h}+2(q-p)\ge a_i+2(h-i)\ge 2(q-i)+1,\]
	where the last step uses $a_i\ge 2(q-h)+1$. Therefore,
		\[\sum_{t=1}^p a_t\ge a_1+(i-1)a_i\ge a_1+(i-1)\ge 2q-i\ge q.\]
	Thus \eqref{sum-a} holds in case (i).
	
	In case (ii), there exists $j\ge 1$ such that $b_j>2r_j$. Set
	\[l=\mathrm{card}\{i:e_i-e_{p+j}\in w\Delta^+(\frg, \frt)\}=p-r_j\ge 1.\]
	Since $e_1-e_{p+j}\in w\Delta^+(\frg, \frt)$, we have $b_j\ge 2(p-l)+1$. Since $e_1-e_{p+1}\in w\Delta^+(\frg, \frt)$, we have
	\[a_1\ge b_1+2(q-p)\ge b_j+2(q-p)\ge  2(q-l)+1,\]
	where the last step uses $b_j\ge 2(p-l)+1$. Therefore,
		\[\sum_{t=1}^p a_t\ge a_1+(l-1)a_l\ge a_1+(l-1)\ge 2q-l\ge q.\]
	Thus \eqref{sum-a} holds in case (ii).
	
By the induction hypothesis, we have
$$
\caD_{p-1}\big((a_2,\dots,a_p), (k_2,\dots,k_p)\big)+\caD_q(b,R_w)\ge 2(p-1)q.
$$
On the other hand,
by \eqref{D_r-2} and \eqref{sum-a}, we have
$$
\caD_{p}(a,K_w)-\caD_{p-1}\big((a_2,\dots,a_p),(k_2,\dots,k_p)\big)\ge 2q,
$$
Adding the two equations above gives \eqref{D_r-3}. This handles cases (i) and (ii).
	
Now let us consider case (iii). Since $\mu$ is u-large, Lemma \ref{large} implies that
$$
a_1\ge 2k_1+1=2q+1.
$$
Therefore, for each $2\le i\le p$,
	\[a_1-a_i\ge q-k_i,\quad a_1+a_i\ge q+k_i.\]
	Hence
	\[\min\{a_1-a_i,q-k_i\}+\min\{a_1+a_i,q+k_i\}=2q.\]
	By \eqref{D_r-1},
	\begin{align*}
		\caD_p(a,K_w)&\ge 2\min\{a_1, k_1\}+\sum_{i=2}^p(\min\{a_1-a_i, q-k_i\}+\min\{a_1+a_i, q+k_i\})\\
		&=2q+2q(p-1)=2pq.
	\end{align*}
	This handles case (iii) and finishes the situation that $w \xi_1=e_1$.
	
	Now one sees that the situation $w\xi_1=e_{p+1}$ is similar and we omit the details.
\end{proof}

\begin{prop}
	Take $\mu$ as in \eqref{mu}.	Assume that  $\mu$ is in $\caU_{p, q}^c$ and that $\mu+2\rho_c\in\caC_w$. Then $B(\mu)\le B_0$.
\end{prop}
\begin{proof}
	Put $P_i=p-i+1$ and $Q_j=q-j+1$. 	Write  $\rho_n^{(w)}$ as in \eqref{rho-n-w}. Note that
	\[\mu+2\rho_c-\rho^{(w)}=(a_1-k_1+P_1,\dots,a_p-k_p+P_p\mid b_1-r_1+Q_1,\dots,b_q-r_q+Q_q),\]
Therefore,
	\begin{align*}
L(\mu) &= \sum_{i=1}^p(a_i-k_i+p-i+1)^2+\sum_{j=1}^q(b_j-r_j+q-j+1)^2-\dfrac{e(\mu)}{2}\\
&= \Phi_p(|a-K_w|) + 2\caD_p(a,K_w) - 4 \caF_p(K_w)\\
&+ \Phi_q(|b-R_w|)+ 2\caD_q(b,R_w) - 4\caF_q(R_w)-\dfrac{e(\mu)}{2},
\end{align*}
where the first step uses \eqref{L-0}, while the second step uses \eqref{Phi-Dr}.
Note that
$$
\|\{\mu-\rho_n^{(w)}\} +\rho_c\|^2= \Phi_p(|a-K_w|) + \Phi_q(|b-R_w|)
$$
Thus
\begin{equation}\label{large-2}
		D_w(\mu)=4\caF_p(K_w)+4\caF_q(R_w)-2\caD_p(a,K_w)-2\caD_q(b,R_w)+\dfrac{e(\mu)}{2}.
	\end{equation}
	For each pair $(i, j)$, if $e_i-e_{p+j}\in w\Delta^+(\frg,\frt)$, then it contributes $1$ to $k_i$, and hence contributes $P_i$ to $\caF_p(K_w)+\caF_q(R_w)$. Similarly, if $e_{p+j}-e_i\in w\Delta^+(\frg,\frt)$, then it contributes $Q_j$ to $\caF_p(K_w)+\caF_q(R_w)$. Therefore,
	\begin{align}\label{large-3}
		\caF_p(K_w)+\caF_q(R_w)&\le \sum_{i=1}^p\sum_{j=1}^q\max\{P_i,Q_j\}\nonumber\\
		&=\sum_{i=1}^p\left(\dfrac{q(q+1)}{2}+\dfrac{P_i(P_i-1)}{2}\right)\nonumber\\
        &=\sum_{i=1}^p\left(\dfrac{q(q+1)}{2}+\dfrac{i(i-1)}{2}\right)\nonumber\\
		&=\dfrac{pq(q+1)}{2}+\dfrac{p(p^2-1)}{6}.
	\end{align}
Thus
	\[B(\mu)\le D_w(\mu)\le 2pq(q-1)+\frac{2}{3} p(p^2-1)+\frac{p}{2}=B_0-4\delta_p^q\le B_0,\]
where the second step use \eqref{D_r-3}, \eqref{large-2} and \eqref{large-3} and the estimation that $e(\mu)\le p$; while the third step uses the definition of $B_0$ in \eqref{B0}.
\end{proof}

This finishes the proof of Theorem \ref{thm-main} for $\mu\in\caU_{p, q}^c$.

\medskip



\begin{thebibliography}{99}
	
\bibitem{ALTV} J.~Adams, M.~van Leeuwen, P.~Trapa, D.~Vogan, \emph{Unitary representations of real reductive groups}, Ast{\'erisque} \textbf{417} (2020).	

	
\bibitem{DD16} J.~Ding, C.-P.~Dong, \emph{Spin norm, $K$-types and tempered representations}, J.  Lie Theory
\textbf{26} (2016), 651--658.
	
\bibitem{D13} C.-P.~Dong,
	\emph{On the Dirac cohomology of complex Lie group representations},
	Transform. Groups \textbf{18} (2013), 61--79. Erratum: Transform.
	Groups \textbf{18} (2013), 595--597.
	
\bibitem{D23}
C.-P.~Dong, \emph{On the Helgason-Johnson bound}, Israel J. Math \textbf{254} (2023), 373--397.
	

\bibitem{DY} C.-P.~Dong, Z.~Ying,
	\emph{Distribution of spin norm along pencils: the $Sp(p, q)$ case},
	preprint 2026, arXiv:2605.06015.
	

	
\bibitem{HJ} S.~Helgason, K.~Johnson, \emph{The bounded spherical functions on symmetric spaces}, Adv. Math.
	\textbf{3} (1969), 586--593.

\bibitem{HP} J.-S.~Huang, P.~Pand\v zi\'c, \emph{Dirac
cohomology, unitary representations and a proof of a conjecture of
Vogan}, J. Amer. Math. Soc.  \textbf{15} (2002), 185--202.



\bibitem{Pa72} R.~Parthasarathy, \emph{Dirac operators and the discrete
series}, Ann. of Math. \textbf{96} (1972), 1--30.

\bibitem{Pa80} R.~Parthasarathy,
\emph{Criteria for the unitarizability of some highest weight
modules}, Proc. Indian Acad. Sci. \textbf{89} (1980), 1--24.

\bibitem{PRV} K.~R.~Parthasarathy, R.~Ranga Rao,
V.~S.~Varadarajan, \emph{Representations of complex semi-simple Lie
groups and Lie algebras}, Ann. of Math. \textbf{85} (1967),
383--429.
	
\bibitem{SV} S.~Salamanca-Riba, D.~Vogan, \emph{On the classification of unitary representations of reductive Lie groups}, Ann. of Math. \textbf{148} (1998), 1067--1133.
	
	
\bibitem {V80} D.~Vogan,
	\emph{Singular unitary representations},  Noncommutative harmonic
	analysis and Lie groups (Marseille, 1980),  506--535.
	
\bibitem {V81} D.~Vogan,
	\emph{Representations of real reductive Lie groups}, Progress in Mathematics, Vol.~\textbf{15}, Birkh\"auser, Boston, 1981.
	
	
\bibitem {V97} D.~Vogan, \emph{Dirac operators and unitary
		representations}, 3 talks at MIT Lie groups seminar, Fall 1997.

\bibitem{V20-1} D.~Vogan,
\emph{Dirac inequality and computing the unitary dual}, \texttt{atlas} seminar, September 22, 2020.
See \url{http://math.mit.edu/~dav/atlassem/}

\bibitem{V20-2} D.~Vogan,
\emph{Finding the sharpest Dirac inequality}, \texttt{atlas} seminar, September 29, 2020.
See \url{http://math.mit.edu/~dav/atlassem/}

\bibitem{W24} K.~D.~Wong, \emph{On some conjectures of the unitary dual of $U(p, q)$}, Adv. Math.
\textbf{442} (2024), Paper No.~109584, 38 pages.

\bibitem{At} Atlas of Lie Groups and Representations, version 1.1.1, May 2026. See \url{www.liegroups.org}
 for more about the software.
	
\end{thebibliography}
\end{document}